\documentclass[reqno, oneside]{amsart}
\usepackage[foot]{amsaddr}

\usepackage{amsmath,amssymb,amsthm}
\usepackage{mathtools}
\usepackage{amsfonts}
\usepackage{physics}

\usepackage[T1]{fontenc}
\usepackage[utf8]{inputenc}
\usepackage[english]{babel}

\usepackage{kpfonts}
\usepackage{microtype}
\usepackage{bm}
\usepackage[style = american]{csquotes}

\usepackage{enumitem}
\usepackage[bookmarksnumbered,linktocpage,hypertexnames=false,colorlinks=true,linkcolor=blue,citecolor=purple,urlcolor=black,anchorcolor=green,breaklinks=true]{hyperref}
\usepackage[doi=false, url=false, isbn=false, giveninits=true,date=year,style=alphabetic,clearlang=true,maxbibnames=99,backend=biber]{biblatex}
\renewbibmacro{in:}{}
\AtEveryBibitem{\clearfield{eprintclass}}
\AtEveryBibitem{\clearfield{url}}

\usepackage{tabularx}
\usepackage{cleveref}
\usepackage{xcolor}
\usepackage{multicol}
\usepackage{centernot}

\newtheorem{theoremintro}{Theorem}

\newtheorem{theorem}{Theorem}[subsection]
\newtheorem{lemma}[theorem]{Lemma}
\newtheorem{proposition}[theorem]{Proposition}
\newtheorem{corollary}[theorem]{Corollary}

\theoremstyle{definition}
\newtheorem{definition}[theorem]{Definition}
\newtheorem{example}[theorem]{Example}
\newtheorem{remark}[theorem]{Remark}

\newcommand{\C}{\mathbb C}
\newcommand{\PP}{\mathbb P}

\newcommand{\R}{\mathbb R}

\def\CC{\mathbb{C}}
\def\NN{\mathbb{N}}
\def\RR{\mathbb{R}}
\def\ZZ{\mathbb{Z}}
\def\KK{\mathbb{K}}
\def\SS{\mathbb{S}}

\renewcommand{\epsilon}{\varepsilon}

\newcommand{\cI}{\mathcal{I}}

\newcommand{\cE}{\mathcal{E}}

\newcommand{\cV}{\mathcal{V}}

\newcommand{\mcone}{\mathcal{M}}

\renewcommand{\div}{\mathrm{div}}

\newcommand{\rr}{\mathrm{rr}}

\DeclareMathOperator{\double}{\mathfrak{D}}
\DeclareMathOperator{\relint}{\mathrm{ri}}
\DeclareMathOperator{\interior}{\mathrm{int}}
\DeclareMathOperator{\mult}{mult}

\DeclareMathOperator{\pos}{\mathrm{P}}

\DeclareMathOperator{\cone}{cone}

\DeclareMathOperator{\quot}{Quot}

\DeclareMathOperator{\jac}{J}
\DeclareMathOperator{\supp}{supp}

\DeclareMathOperator{\Cl}{Cl}
\DeclareMathOperator{\Div}{Div}

\def \eval {\mathrm{ev}}
\def \car {\mathcal{C}}
\def \face {\mathrm{F}}
\def \faces {\mathcal{F}}

\setlist[enumerate,1]{label={(\roman*)},ref={\thetheorem (\roman*)}}

\DeclarePairedDelimiterX{\BW}[2]{\langle}{\rangle_{BW}}{#1, #2}

\usepackage{lineno}
 \usepackage[normalem]{ulem}
\usepackage{comment,soul,fullpage}
\title%
{Nonnegative Polynomials and Moment Problems on Algebraic Curves}

\author{Lorenzo Baldi\textsuperscript{*} \and Grigoriy Blekherman \textsuperscript{$\dagger$} \and Rainer Sinn\textsuperscript{$\ddagger$}}
\address{\textsuperscript{*}Sorbonne Universit\'e, CNRS, F-75005, Paris, France, and MPI MiS, Leipzig, Germany}
\address{\textsuperscript{$\dagger$}Georgia  Institute  of  Technology,  Atlanta,  USA}
\address{\textsuperscript{$\ddagger$}Universität Leipzig, Leipzig, Germany}
\email{lorenzo.baldi@mis.mpg.de}
\date{\today}

\begin{document}

\begin{abstract}
The cone of nonnegative polynomials is of fundamental importance in real algebraic geometry, but its facial structure is understood in very few cases. We initiate a systematic study of the facial structure of the cone of nonnegative polynomials $\pos$ on a smooth real projective curve $X$. We show that there is a duality between its faces and totally real effective divisors on $X$. This allows us to fully describe the face lattice in case $X$ has genus one. We compute the Carath\'{e}odory number of the dual moment cone $\pos^\vee$ for an elliptic normal curve $X$, which measures the complexity of quadrature rules of measures supported on $X$. Interestingly, the topology of the real locus of $X$ influences the Carath\'{e}odory number of $\pos^\vee$. We apply our results to truncated moment problems on affine cubic curves, where we deduce sharp bounds on the flat extension degree.
\end{abstract}

\maketitle

\section{Introduction}

Providing certificates of positivity for real nonnegative polynomials is a fundamental challenge in real algebraic geometry, dating back to the works of Hilbert \cite{hilbertUeberDarstellungDefiniter1888} and Artin \cite{artinUberZerlegungDefiniter1927} on the existence of sums of squares representations. Over the last century, these results, now called Positivstellens\"atze, have been extensively studied and generalized, see e.g. \cite{marshallPositivePolynomialsSums2008a}. However, the geometric properties of the set of nonnegative polynomials are less understood. The polynomials of degree $2d$ nonnegative on a variety $X$ form a convex cone, which we denote by $\pos_{X,2d}$; its face structure is fully known only in very few cases \cite{schulzeConesLocallyNonNegative2021, kunertFacialStructureCones2014}: degree $2d$ univariate polynomials (or equivalently, quadratic polynomials on the rational normal curve), quadratic polynomials on $\R^n$ and quartic polynomials on $\RR^2$.
In this work, we initiate a systematic study of the facial structure of the cones of nonnegative polynomials on a real projective curve $X$, and in particular of the extreme rays of these cones. For elliptic normal curves, we provide a full description of the face lattice.

While the cone of nonnegative polynomials is a central object in real algebraic geometry, its dual convex cone $\pos_{X,2d}^\vee$ plays an important role in real analysis. In the analysis literature, $\pos_{X,2d}^\vee$ is called the moment cone \cite{schmudgenMomentProblem2017a}. It corresponds to the convex cone of linear functionals, acting on polynomials of degree $2d$, that can be written as integration with respect to a measure supported on $X$. We leverage our understanding of the cone of nonnegative polynomials on elliptic normal curves to find the Carath\'{e}odory number of the dual moment cone, refining results of \cite{didioMultidimensionalTruncatedMoment2021a}. Interestingly, we find that the Carath\'eodory number depends on the topology of the real locus: smooth totally real, genus one curves can have either one or two connected components, and the Carath\'{e}odory number is different in these two cases. 

Our main results deal with projective curves, whereas truncated moment problems are usually studied in the affine setting. We apply projective study to investigate the truncated moment problem on smooth affine plane cubic curves. Such truncated moment problems have been studied for specific rational curves, e.g.~in \cite{fialkowSolutionTruncatedMoment2011,zalarTruncatedMomentProblem2023}. We provide sharp degree bounds for flat extension, improving the known results and completing the picture for the truncated moment problem on smooth cubic planar curves.
\subsection{Main results and related works}
Let $X\subset \PP^n = \PP^n(\CC)$ be a real projective curve, and denote the real locus of $X$ by $X(\RR) \subset \PP^n(\RR)$.
We work with totally real curves, i.e. curves such that $X(\RR)$ is Zariski dense in $X$. Our main object of study is the cone of nonnegative forms on $X$ of degree $2d$. By replacing $X$ with the $d$-th Veronese embedding $\nu_{n, d}(X)$ we may restrict ourselves to analyzing nonnegative quadratic polynomials on real curves, and thus we define
 \[\pos_{X, 2} \coloneqq \big\{ \, q \mid \deg q = 2 \text{ and for all } A \in X(\RR), \ q(A) \ge 0 \, \big\}\] 
to be the convex cone of quadratic forms nonnegative on $X(\RR)$. $\pos_{X, 2}$ is a closed, pointed convex cone in the real vector space $R_2$ of quadratic forms on $X$.  
In \Cref{sec:psd_forms_curves}, we study the face structure of $\pos_{X,2}$. The cone $\pos_{X,2}$ has one zero-dimensional face: the origin. One-dimensional faces of $\pos_{X,2}$ are its extreme rays. To any positive dimensional face $\face$ we associate a unique totally real effective divisor $\div(\face)$, which we call the \emph{face divisor} of $\face$, as follows (we refer to \Cref{sec:faces_to_divisors} for more details). The face divisor is equal to half the real part of the divisor $\div q$, for any quadric $q$ in the relative interior of $\face$. In the converse direction, to any totally real effective divisor $D$ we associate the face $\face_D$ consisting of all quadratic forms $q \in \pos_{X,2}$ such that $\div q \geq 2D$.The above allows us to define a two maps $\Phi$ and $\Psi$ between faces $\face \subset \pos$ and totally real effective divisors on $X$: $\Phi$ which sends a positive dimensional face of $\pos_{X,2}$ to its face divisor $\div(\face)$, and the map $\Psi$ which sends a totally real effective divisor $D$ to its associated face $\face_D$.
\begin{theoremintro}[{see \Cref{thm:char_faces} and \Cref{cor:galois}}]
    \label{thm:A}
    Let $(\faces\setminus\{\, \{0\}\, \}, \subset)$ be the poset of positive-dimensional faces of $\pos_{X,2}$, ordered by inclusion. Then
    $\Phi$ and $\Psi$ form a Galois connection between $(\faces\setminus\{\, \{0\}\, \}, \, \subset \,)$ and the face divisors $(\Im \Phi, \le)$. This means that, for all positive dimensional faces $\face \in \faces\setminus\{\, \{0\}\, \}$ and face divisors $D \in \Im \Phi$:
    \[
        \div(\face) \le D \iff \face \supseteq \face_{D}
    \]
    Moreover, we have $(\Psi \circ \Phi)(\face) = \face_{\div(\face)} = \face$ for all $\face \in \faces\setminus\{\, \{0\}\, \}$.
\end{theoremintro}

The study of the face lattice of $\pos$ thus reduces to understanding the following crucial question: \emph{which totally real effective divisors are face divisors of some face $\face$ of $\pos$?} In \Cref{sec:dimension_normality}, we study the dimension of the faces of $\pos$ using Riemann-Roch. In \Cref{sec:extreme_rays}, we focus in particular on the extreme rays of $\pos_{X, 2d}$, which we denote by $\cE(\pos_{X, 2d})$. We introduce a map $\double \colon \cE(\pos) \to \jac$ from the set of extreme rays to the Jacobian $\jac$ of the curve (see \Cref{sec:double}), and focus in \Cref{sec:real_rooted} on the set $\cE^{\rr}(\pos_{X, 2d})$ of extreme rays spanned by forms with only real zeroes on $X$. We characterize the image of $\cE^{\rr}(\pos_{X, 2d})$ under $\double$ as the set of positive 2-torsion points in $J(\RR)_2^+$ (see \Cref{def:positive_torsion}), which is an elementary abelian group of order $2^g$, and study its fibers.

\begin{theoremintro}[{see \Cref{cor:families_torsion}}]\label{thm:B}
    Let $X\subset\PP^n$ be a totally real smooth irreducible curve of genus $g$. With the previous notations, for all sufficiently large $d$, \[\double(\cE^{\rr}(\pos_{X, 2d})) = \jac(\RR)_2^+ \cong (\ZZ/2\ZZ)^{g}\]
    Moreover, $\double(\RR_{\ge 0} \cdot q_1) = \double(\RR_{\ge 0} \cdot q_2)$ if and only if there exists $g \in \RR(X)$ such that $q_1 = g^2 q_2$.
\end{theoremintro}

\Cref{thm:B} has surprising consequences for the representation of positive polynomials, as it predicts the existence of forms $q \in \pos_{X, 2d}$ such that $q^{2k+1}$ is not a sum of squares for all $k \in \NN$ \cite{baldi2026stubbornpolynomials}. Using \Cref{thm:B}, we investigate in \Cref{sec:genus_one} the case of genus one curves, when we have $\cE(\pos_{X, 2d}) = \cE^\rr(\pos_{X, 2d})$ (see also \cite{kummer2026nonnegativepolynomialsgeneralizedellipticv3}). The special case of plane cubics, which was our original motivation, is described using elementary techniques in \Cref{sec:plane_cubics}. 
More generally, the complete facial structure of $\pos_{X, 2d}$ for elliptic normal curves is described in \Cref{thm:elliptic_normal}.

Extreme rays of positivity cones have been previously investigated in the literature. In \cite{choiExtremalPositiveSemidefinite1977}, the authors provide several examples of extreme rays (that are not sums of squares) of the nonnegative cone for the case $X = \PP^n$. Additional examples can be found in \cite{reznickConcreteAspectsHilbert2000}. We refer to \cite{schulzeConesLocallyNonNegative2021} for the description of Hilbert's cases and other results of local nature. The case of ternary sextics was studied recently in \cite{kunertExtremePositiveTernary2018}, where the authors characterize the sets of nine points in the projective plane that admit a nonnegative, extreme ternary sextic vanishing on them.
But these results are only partial, and before \Cref{thm:elliptic_normal} a complete description of the faces of the nonnegative cone was known to be possible only when the nonnegative cone coincides with the sums of squares cone.

In \Cref{sec:cara}, we move to the study of the dual cone $\pos_{X, 2}^\vee$, which consists of linear functionals in the dual space to quadrics on $X$, that can be expressed as integration with respect to a measure supported on $X(\RR)$. Equivalently, $\pos^\vee$ is the convex hull of the set of point evaluations of quadrics at the real points $X(\RR)$ of $X$. The \emph{Carath\'eodory number} $\car_{X,2}$ of $\pos_{X, 2}^\vee $ is the minimal natural number $\car_{X,2}$ such that any linear functional $L \in \pos_{X, 2}^\vee $ is conic sum of at most $\car_{X,2}$ point evaluations. This important quantity has been studied from different perspectives, and it is equivalent to determining the minimal number of nodes in a quadrature rule, or the maximal rank in a symmetric tensor decomposition with nonnegative coefficients. Our study of the nonnegative cone on elliptic normal curves allows us to determine the Carath\'eodory number of $\pos_{X, 2}^\vee$, or more generally of $\pos_{X, 2d}^\vee$, which remarkably depends on the topology of the real locus of the curve.

\begin{theoremintro}[{see \Cref{thm:car}}]\label{thm:D}
Let $X \subset \PP^n$ be a totally real elliptic normal curve.
\begin{enumerate}
    \item If $X(\RR)$ is connected, then $\car_{X,2} = \deg X = n+1$.
    \item If $X(\RR)$ is disconnected, then $\car_{X,2} = \deg X +1 = n+2$.
\end{enumerate}
\end{theoremintro}
The proof of \Cref{thm:D}, developed in \Cref{sec:critical} is inspired by Hilbert's proof that every nonnegative ternary quartic is a sum of at most three squares (see \cite{blekhermanLowRankSumofSquaresRepresentations2019} for a modern exposition). It is natural to ask if the same proof technique can be applied to curves of higher genus as well.

Exact results concerning Carath\'eodory numbers are rare in the literature. The rational univariate case is e.g. completely solved in \cite[Cor.~9.16]{schmudgenMomentProblem2017a}. To the best of our knowledge \Cref{thm:D} is the first complete description for non-rational curves, and the first time where the topology of the real supporting set is shown to play a key role in the study of Carath\'eodory numbers. \Cref{thm:D} distinguishes between the two possibilities left open from \cite[Th.~4.8]{didioMultidimensionalTruncatedMoment2021a}.
Asymptotic estimates for Carath\'eodory numbers on affine curves and for higher dimensional cases have been recently studied in \cite{rienerOptimizationApproachesQuadrature2018a,didioMultidimensionalTruncatedMoment2018, didioMultidimensionalTruncatedMoment2021a}.

We can also interpret the Carath\'eodory number as the rank of a real symmetric tensor decomposition with nonnegative coefficients, using forms supported on the Veronese embedding of the curve (we refer to \Cref{sec:waring} for more details). The similarities and differences between the complex case, real case, and real case with nonnegative coefficients have been investigated in \cite{qiSemialgebraicGeometryNonnegative2016,blekhermanRealRankRespect2016, angeliniRealIdentifiabilityVs2018}. 

We now describe our results on the truncated moment problem, contained in \Cref{sec:dual_plane_cubics}. We switch from the projective setting to the affine one, as it is more common in the analysis literature, and focus on the planar case.
Let then $X$ be an affine real cubic plane curve, whose projectivization is smooth. 
The \emph{moment problem} asks whether a liner functional $L$ can be written as integration with respect to a measure supported on $X(\RR)$. Unlike in the projective case, this is not equivalent to $L$ lying in the dual cone $\pos_{X,\le 2d}^\vee$, since the cone of linear functionals that are representable by measures is not always closed. The most effective criterion for solving this problem is the \emph{flat extension} criterion, which can be stated as follows. Given a linear functional $L$ acting on polynomials on degree $\le2d$, we say that $\widetilde L$, acting on polynomials of degree $\le 2d+2$, is a flat extension of $L$ if:
\begin{enumerate}
    \item $L$ is the restriction of $\widetilde L$ to polynomials of degree $\le 2d$; and
    \item the bilinear forms $(p,q) \mapsto L(pq)$ and $(p,q) \mapsto \widetilde L(pq)$ (defined respectively on polynomials of degree $\le d$ and $\le d+1$) have the same rank.
\end{enumerate}
The flat extension criterion states that if $L$ has a flat extension, then $L \in \pos_{X,\le 2d}^\vee$.

Leveraging our projective study of the Carath\'eodory numbers, we show that for plane cubics the moment problem is conveniently described via the existence of an \emph{almost flat extension} of $L$. This is an extension $\widetilde L$ of $L$ to degree $\le 2d+4$, where the rank of the associate bilinear form is allowed to grow by one. We refer to \Cref{sec:almost_flat} for more details. Note that the existence of an almost flat extension guarantees a flat extension either between degree $2d$ and $2d+2$, or between degree $2d+2$ or $2d+4$.

\begin{theoremintro}[{see \Cref{prop:almost_flat}}]\label{thm:E}
        Let $X(\RR) \subset \RR^2$ be the real locus of a totally real affine plane cubic, whose projectivization $\overline{X}$ is smooth, and let $L$ be a linear functional acting on $\RR[x, y]_{\le 2d}\big/ \cI(X)_{\le 2d}$. Then $L$ is a moment linear functional if and only if $L$ admits an almost flat extension.
\end{theoremintro}
The proof of \Cref{thm:E} is given in \Cref{sec:moment_problem_for_cubics}. We provide refined versions of this result in \Cref{thm:flat_conn_one_infty,thm:almost_flat}, which take into account the topology of the curve in analogy to the projective case.
To the best of our knowledge, this is the first solution of the moment problem for irreducible, nonrational curves.

The truncated moment problem for plane cobic curves was previously studied, e.g., in \cite{fialkowSolutionTruncatedMoment2011, zalarTruncatedMomentProblem2023, zalarTruncatedMomentProblem2022} for the union of parallel lines. All these works deal with (unions of) rational curves, and the authors are able to characterize membership in the moment cone with the flat extension condition. The flat extension condition is not sufficient in our genus one case, and to overcome this problem we introduce the notion of almost flat extension. Moreover, we remark that these results apply for linear functionals $L\in \pos_{X,\le 2d}^\vee$ for large enough $d$, while \Cref{thm:E} applies for any $d\ge 1$. A complete solution of the truncated moment problem for plane cubics has recently appeared in \cite{kummer2026nonnegativepolynomialsgeneralizedellipticv1}, extending \Cref{thm:E} to singular and reducible curves.

In this work, we restricted ourselves to studying positivity cones inside $\RR[X]_{2d}$, to keep the article accessible also to researchers in functional analysis. However, analogous results apply for the convex cones of nonnegative sections in $H^0(X, L \otimes L)$, where $L = \mathcal{O}_X(D)$ is an (ample, real) line bundle on $X$. This is the point of view taken e.g. in \cite{kummer2026nonnegativepolynomialsgeneralizedellipticv3} and in \cite[Sec.~5]{blekhermanSumsSquaresVarieties2016a}.

\section{Preliminaries and notations}\label{sec:preliminaries}
\subsection{Real algebraic curves}
\label{sec:real_curves}
We refer to \cite{mangolteRealAlgebraicVarieties2020a} for the basics of real algebraic geometry, which we briefly recall in this section. In the following, varieties are always to be understood as \emph{reduced}, \emph{(geometrically) irreducible} and \emph{smooth}.

An \emph{$\RR$-variety} is a pair $(X, \sigma)$, where $X$ is an (abstract) complex algebraic variety and $\sigma$ is an anti-regular (or anti-holomorphic) involution on $X$, see \cite[Def.~2.1.12]{mangolteRealAlgebraicVarieties2020a}. We will often omit $\sigma$ from the notation. The \emph{real locus} of the $\RR$-variety $(X,\sigma)$, denoted $X(\RR)$, is the set of fixed points for $\sigma$, i.e. $X(\RR) \coloneqq \{\, x\in X \colon \sigma(x)=x \,\}$. We denote by $\CC(X)$ the $\CC$-algebra of rational functions and by $\RR(X)$ the $\RR$-algebra of real rational functions on $X$. We say that an $\RR$-variety is \emph{totally real} if $X(\RR)$ is Zariski dense in $X$ (these are called varieties with \emph{enough real points} in \cite{mangolteRealAlgebraicVarieties2020a}).

Concretely, we are interested in \emph{real algebraic projective subvarieties}, which we call \emph{real varieties} for short.  
These are subvarieties $X \subset \PP^n = \PP^n(\CC)$ of the projective space, which are invariant under the natural involution of $\PP^n$, i.e. the coordinate-wise conjugation. Equivalently, real varieties $X\subset \PP^n$ are given as the zero locus of (finitely many) real $n$-variate homogeneous polynomials $\{ \, p_1 , \dots , p_a \}  \subset \RR[x_0, \dots , x_n] = \RR[\vb x]$.
The real homogeneous coordinate ring is $\RR[X] = \RR[\vb x]/I$, where $I=\cI(X)$ is the ideal of (real) polynomials vanishing on $X\subset \PP^n$. $\RR[X]$ is a graded ring, with the grading induced by $\RR[\vb x] = \RR[\PP^n]$. We denote by $R_k \coloneqq \RR[X]_k = \RR[\vb x]_k/I_k$ the $k$-graded part, i.e. the $\RR$-vector space of homogeneous real polynomials (or \emph{forms}) of degree $k$ on $X(\RR)$. 
The field of real rational functions on $X$ is $\RR(X) = \quot (\RR[U])$, where $U\subset X$ is an open real affine variety and $\RR[U]$ is the coordinate ring of $U$.

If $\dim X = 1$ (dimension as a complex algebraic variety), we say that $(X, \sigma)$ (resp. $X \subset \PP^n$) is a real curve (resp. a real projective curve). When $X$ is totally real, the real locus $X(\RR)$ of a real curve is a real differentiable manifold of dimension $1$.

A \emph{(Weil) divisor} $D = \sum_{A \in X} a_A A$ on a curve $X$ is an element of the free abelian group on the points of $X$. This means that a divisor is a formal combination of points $A \in X$ with integer coefficients $a_A \in \ZZ$, where $a_A = 0$ for all but finitely many $A \in X$. We denote by $(\Div{X},+)$ the abelian group of divisors on $X$. A divisor $D = \sum_{A \in X} a_A A \in \Div X$ is called \emph{effective} if $a_A \ge 0$ for all $A$. The \emph{support} of $D$, denoted $\supp D$, is the set of $A \in X$ such that $a_A \neq 0$.

We say that a divisor $D = \sum_{A \in X} a_A A$ on a real curve $X$ is \emph{real} if it is invariant under $\sigma$, i.e. if $\sum_{A \in X} a_A \, A = \sum_{A \in X} a_A \, \sigma(A)$. We say that $D = \sum_{A \in X} a_A \, A$ is \emph{totally real} if the support of $D$ is contained in $X(\RR)$, i.e. if $a_A =0$ for all $A \in X \setminus X(\RR)$. We are particularly interested in totally real effective divisors, i.e. divisors with nonnegative integer coefficients whose support is included in $X(\RR)$. Given a divisor $D = \sum_{A \in X} a_A A$, we write $D_{\RR} = \sum_{A \in X(\RR)} a_A A$ and $D_{\CC} = \sum_{A \in X\setminus X(\RR)} a_A A$. Then $D$ is totally real if and only if $D = D_{\RR}$.

\subsection{Convex geometry}
\label{sec:convex_geom}
We refer to \cite{rockafellarConvexAnalysis1970a, barvinokCourseConvexity2002} for the basics of convex geometry.

Given a finite dimensional vector space $V$, we call $\mathrm{C} \subset V$ a \emph{convex cone} if $\RR_{\ge 0} \cdot \mathrm{C} \subset \mathrm{C}$ and $\mathrm{C} + \mathrm{C} \subset \mathrm{C}$. Given $B \subset V$, we denote by $\cone (B)$ the smallest convex cone containing $B$, called the \emph{conic hull} of $B$. We say that a convex cone $C$ is \emph{closed} if it is closed as a subset of $V$ with the Euclidean topology. We say that $C$ is \emph{pointed} if $C \cap -C = \{ 0\}$. In this paper, we will mostly work with convex cones inside $R_2 = \RR[X]_2$, the real vector space of quadratic forms restricted to a totally real projective curve $X \subset \PP^n$.

A \emph{face} of a convex cone $\mathrm{C}$ is a {convex cone} $\mathrm{F} \subset \mathrm{C}$ such that, for all $a, b\in \mathrm{C}$, $a+b \in \mathrm{F}$ implies $a\in \mathrm{F}$ and $b \in \mathrm{F}$. The \emph{dimension} of a face $\mathrm{F}$, denoted by $\dim \mathrm{F}$, is the dimension of the smallest vector subspace $W\subset V$ containing $\mathrm{F}$, called the \emph{linear span} of $\mathrm{F}$. Faces of dimension one are called \emph{extreme rays} of $\mathrm{C}$. A closed {pointed} convex cone is equal to the conic hull (or convex hull) of its extreme rays, see e.g. \cite[Th.~18.5]{rockafellarConvexAnalysis1970a}.

The \emph{relative interior} of a face $\mathrm{F}$, denoted $\relint \mathrm{F}$, is the interior of $\mathrm{F}$ when regarded as a subset of its linear span, equipped with the Euclidean topology. If, given two faces $\mathrm{F}_1, \mathrm{F}_2 \subset \mathrm{C}$, there exists $a \in \relint \mathrm{F}_1 \cap \relint \mathrm{F}_2$, then $\mathrm{F}_1 = \mathrm{F}_2$ (see \cite[Th.~18.1.2]{rockafellarConvexAnalysis1970a}).

Given a real vector space $V$, we denote by $V^*\coloneqq\hom_{\RR}(V, \RR)$ the dual space of $\RR$-linear functions. If $\mathrm{C} \subset V$ is a convex cone, $\mathrm{C}^\vee \coloneqq \{ \, L\in V^* \colon L(V) \subset \RR_{\ge 0} \, \}$ is the \emph{dual cone} to $\mathrm{C}$. A face $\mathrm{F} \subset \mathrm{C}$ is called \emph{exposed} if $\mathrm{F} = L^{-1}(0) \cap C$ for some $L \in \mathrm{C}^\vee$. \emph{Conic duality} states that, given a \emph{closed} convex cone $\mathrm{C}$, we have $(\mathrm{C}^\vee)^\vee = \mathrm{C}$.
\subsection{Nonnegative forms and the moment problem}
\label{sec:moment_prob}
We refer to \cite{blekhermanSemidefiniteOptimizationConvex2012, blekhermanSumsSquaresVarieties2016a, schmudgenMomentProblem2017a} for the study of cones of nonnegative forms and their dual cones.

In the following, we consider a totally real projective subvariety $X \subset \PP^n$.
If $k=2d$ is even, every $q\in R_{2d} = \RR[X]_{2d}$ has a well-defined sign on every $A \in X(\RR)$. We write $q(A)\ge0$ (resp. $q(A)\le 0$) if the sign of $q$ at $A$ is $0$ or positive (resp. negative). If the form $q$ is such that $q(A) \ge 0$ for all $A\in X(\RR)$, we say that $q$ is \emph{nonnegative}. We define:
\begin{equation}
    \pos_{X,2d} \coloneqq \big\{ \, q \in R_{2d} \mid \forall A \in X(\RR) \ q(A) \ge 0 \, \big\} = \big\{ \, q \in R_{2d} \mid q \text{ is nonnegative} \, \big\}\notag
\end{equation}
\begin{equation}
    \Sigma_{X,2d} \coloneqq \left\{ \, \sum_{i=1}^r p_i^2 \in R_{2d} \mid r\in \NN, \ p_1, \dots p_r \in R_d \, \right\}\notag
\end{equation}
The elements of $\Sigma_{X,2d}$ are called \emph{sums of squares} forms.
Clearly, every sum of squares is nonnegative, i.e. $\Sigma_{X,2d} \subset \pos_{X, 2d}$. When clear from the context, we will write $\pos \coloneqq \pos_{X,2d}$ and $\Sigma \coloneqq\Sigma_{X,2d}$. $\pos$ and $\Sigma$ are full dimensional, closed, pointed convex cones in the vector space $R_{2d}$, see \cite{blekhermanSumsSquaresVarieties2016a}.

The dual convex cone $\pos^\vee$ can be characterised as follows. Denote by $\widehat{X}(\RR) \subset \RR^{n+1}$ the affine cone over $X(\RR)\subset \PP^n(\RR)$, and by $\SS^n$ the unit sphere in $\RR^{n+1}$. For every $x \in \widehat{X}(\RR)$, denote by $\eval_x \in R_{2d}^*$
the point evaluation at $x$, i.e. $\eval_x(q)\coloneqq q(x)$ for $q \in  R_{2d}$. Then:
\begin{equation}
    \pos^\vee = \pos_{X,2d}^\vee = \cone \left( \eval_A \colon A \in \widehat{X}(\RR) \right) \notag
\end{equation}
see \cite{blekhermanSumsSquaresVarieties2016a}. We will often replace $\widehat{X}(\RR)$ by ${X}(\RR)$, writing $\eval_A$ for $A \in X(\RR)$, identifying $A$ with any of its affine representatives in $\widehat{X}(\RR) \cap \SS^n$.
The \emph{moment problem} is the problem of determining, given $L \in R_{2d}^*$, whether $L \in \pos^\vee$ or $L\notin \pos^\vee$.

If $L \in \pos^\vee$, we define the \emph{Carath\'eodory number} of $L$ as:
\[
    \car_{X, 2d}(L) \coloneqq \min \big\{ \, r\in \NN \mid \exists x_1, \dots x_r \in \widehat{X}(\RR) \text{ s.t. } L \in \cone(\eval_{x_1}, \dots , \eval_{x_r}) \,  \big\}
\]
and the \emph{Carath\'eodory number} of $\pos^\vee$ as \[\car_{X, 2d} \coloneqq \max \big\{ \, \car_{X, 2d}(L) \colon L \in \pos^\vee \, \big\}\]
The main goal of this article is to study in detail the convex cones $\pos_{X,2d}$, $\pos_{X,2d}^\vee$, to determine the Carath\'eodory number $\car_{X,2d}$ when $X \subset \PP^n$ is a totally real elliptic normal curve, and finally to apply these results to solve the moment problem on affine plane cubics.
\section{The convex cone of nonnegative forms on projective curves}\label{sec:psd_forms_curves}
In the following, $X \subset \PP^n$ is a smooth, irreducible, totally real projective curve. We want to study the convex cone $\pos_{X,2d} \subset R_{2d} = \RR[X]_{2d}$ of nonnegative forms on $X(\RR)$ of degree $2d$.
We can restrict to the case $d=1$, since $\RR[X]_{2d} \cong \RR[\nu_{n,d}(X)]_2$, where $\nu_{n,d}(X)$ denotes the $d$-th Veronese embedding of $X \subset \PP^n$.

We then fix $\pos = \pos_{X,2}$. Our goal is to understand the relationship between faces $\face \subset \pos$ and totally real effective divisors $D \in \Div X$. Recall that, given any divisor $D \in \Div(X)$, we can uniquely write $D = D_\RR + D_\CC$, where $\supp D_{\RR} \subset X(\RR)$ and $\supp D_\CC \subset X \setminus X(\RR)$, and a divisor $D$ is totally real if and only if $D = D_{\RR}$.

\subsection{Bivariate forms}\label{sec:P1}
Before developing the general theory, we describe the special case of $\PP^1$. For a bivariate form $q \in \RR[x_0, x_1]_{2d}$, we denote $\div q$ its divisor of zeroes and write $\div q = (\div q)_\RR + (\div q)_\CC$, where the support of  $(\div q)_\RR$ and $(\div q)_\CC$ is included in $\PP^1(\RR)$ and $\PP^1 \setminus \PP^1(\RR)$ respectively.
\begin{example}[{see also \cite[Prop.~1.4.4]{schulzeConesLocallyNonNegative2021}}]\label{ex:P1} 
    Let $\pos = \pos_{\PP^1, 2d}$ be the convex cone of nonnegative bivariate homogeneous polynomials, or bivariate forms, of degree $2d$. Every nonnegative $q \in \pos$ has all zeroes with even multiplicity on $\PP^1(\RR)$ (otherwise $q$ would change sign).
    We therefore have $(\div q)_\RR = 2D$, for some totally real effective divisor $D$ of degree $\le d$.
    We can determine the faces of $\pos$ by specifying real zeroes on $\PP^1$, as follows: given a totally real effective divisor $D$ of degree $\le d$, we define
    \[
        \face_{D} \coloneqq \left\{ \, q \in \pos \mid \div q \ge 2D \, \right\}
    \]
    These are faces of $\pos$. By divisibility properties of univariate polynomials, every divisor $D$ imposes $2\deg D$ conditions in the $2d+1$ dimensional space of bivariate forms $\RR[x_0, x_1]_{2d}$. More precisely, we have
    $\dim \face_{D} = 2(d - \deg D) + 1$, as we show in the following.
    \begin{itemize}
        \item If $D = 0$ is the empty divisor, then $\face_D = \pos$, and we want to show that $\dim \pos = 2d+1$. Therefore we only need to show that $\pos$ has nonempty interior. But this is clear, as e.g. the form $x_0^{2d} + x_1^{2d}$ is strictly positive on $\PP^1(\RR)$, and it can be perturbed in any direction while remaining strictly positive.
        \item If $D = A$ for some $A \in \PP^1(\RR)$, then $\face_D = \face_A$ consists on nonnegative forms vanishing on $A$. This imposes two conditions on form of degree $2d$, and $\face_D$ is full dimensional in its $2(d-1)+1$ linear span given by
        $V_D \coloneqq \{ \, q \in \RR[x_0, x_1]_{2d} \colon \mult_A(q) \ge 2 \, \}$.
        For instance, if $A = (a:b)$, then
        the form $(bx_0 - ax_1)^2(x_0^{2d-2} + x_1^{2d-2})$ lies in the relative interior of $\face_{D}$, as it vanishes exactly with multiplicity $2$ at $A$ on all $\PP^1(\RR)$: $\div ((bx_0 - ax_1)^2(x_0^{2d-2} + x_1^{2d-2}))_\RR = 2A$.
        \item In general, if $D$ is a totally real effective divisor with $\deg D \le d$, then $\face_D$ is full dimensional in $V_D = \{ \, q \in \RR[x_0, x_1]_{2d} \colon \div q \ge 2D \, \}$, which is a $2(d-\deg D)+1$ dimensional vector space. A point in the relative interior of $\face_D$ can be constructed as follows. If $D = A_1 + \ldots + A_k$ for points $A_i = (a_i:b_i)\in \PP^1(\RR)$, set $f = \prod_{i=1}^k (b_i x_0 - a_i x_1)^2$. Then $f(x_0^{2d-2k} + x_1^{2d-2k})$ lies in the relative interior of $\face_D\subset V_D$.
    \end{itemize}
    All the faces of $\pos$ are of the form $\face_D$ for some totally real effective divisor $D$. Indeed, if $\face \subset \pos$ is a face, then pick $q \in \relint \face$, and write $(\div q)_\RR = 2D$. It is then possible to show that $q \in \relint \face_{D}$, and therefore $\face = \face_D$ from \cite[Cor. 18.1.2]{rockafellarConvexAnalysis1970a} (see also \Cref{lem::rel_int}).
\end{example}

In the next sections, we will generalize \Cref{ex:P1} to arbitrary totally real projective curves $X\subset \PP^n$.

\subsection{Faces and divisors}\label{sec:faces_to_divisors}

 Let $X \subset \PP^n$ be a smooth, irreducible, totally real projective curve. For $0 \neq q \in \CC[X]$ we consider the divisor of $q$, see e.g. \cite[p.~152]{shafarevichBasicAlgebraicGeometry2013}, denoted $\div q$.
This divisor encodes the intersection points (with multiplicity) of $X$ inside $\PP^n$ with the hypersurface defined by any $Q\in \C[\vb x]$ such that $Q + I(X) = q$ in $\C[X]$: $\div q = \sum_{A\in X} \mult_A(Q) \, A$.
We have $\deg (\div q) = \sum_{A\in X} \mult_A(Q) = \deg q \cdot \deg X$.

We now consider $q \in \pos \subset R_2 = \RR[X]_2$. Since $q$ does not change sign on $X(\RR)$, it does not change sign on any connected component of $X(\RR)$. Therefore, if $q$ has a zero at $A \in X(\RR)$, the multiplicity of $q$ at $A$ is even. This implies that $(\div q)_{\RR} = 2D$ for some effective totally real divisor $D$.

We now show that any non-zero face $\face \subset \pos$ determines a unique totally real effective divisor $D$ on $X \subset \PP^n$.
\begin{lemma}\label{lem::rel_int}
    Let $\{ 0 \} \neq \face \subset \pos$ be a face. Then there exists a unique totally real effective divisor $D \subset \Div(X)$ such that, for all $q$ in the relative interior of $\face$, $(\div q)_\RR = 2D$.
\end{lemma}
\begin{proof}
    For $A \in X(\RR)$, let $\mu_A = \min \{ \, \mult_A(q) \colon q \in \face \,\}/2$. Note that, since $\face \neq \emptyset$, we have $\mu_A \neq 0$ only for finitely many $A\in X(\RR)$ (any quadratic form $q\in \CC[X]$, $q\neq 0$, only vanishes at finitely many points on $X$). We show that, for all $q \in \face$ and $A \in X(\RR)$, $\mult_A(q) > 2\mu_A$ implies $q \notin \relint \face$. So we define the totally real effective divisor $D \coloneqq \sum_{A\in X(\RR)} \mu_A \, A$. The following argument then shows $(\div q)_\RR = 2D$ for all $q \in \relint \face$ as claimed.

    Consider the vector space $V_A = \{ \, q \in R_2 \colon \mult_A(q) \ge 2\mu_A \, \}$. By definition, $\face \subset V_A$.
    We construct a linear functional $L \colon V_A \to \RR$ as follows.
    Consider a local parameter $t$ at $A \in X(\RR)$, and expand $q \in V_{A}$ locally around $A$ so that $q = a_q t^{2 \mu_A} + \dots$, where the dots represent higher order terms in $t$. We define $L(q) = a_q$. Since every $q\in \face$ is nonnegative on $X(\RR)$ we conclude $L(q) \ge 0$ for all $q \in \face$, i.e. $L\in \face^\vee \subset V_A^*$. By definition of $\mu_A$, there exists a form $g \in \face$ with $\mult_A(g) = 2\mu_A$, which implies $L(g)>0$. So $L$ defines a supporting hyperplane to $F$ containing any $q\in F$ with $\mult_A(q) > 2\mu_A$, showing $q \notin \relint \face$.
\end{proof}

Since the divisor $(\div q)_\RR$ is equal for all $q\in \relint \face$ by \Cref{lem::rel_int} we can make the following definition.
\begin{definition}\label{def:face_divisor}
    Given a face $\{ \,\{ 0 \} \, \} \neq \face \subset \pos$, we define the \emph{face divisor} of $\face$ as the unique totally real divisor $\div(\face)$ such that $(\div q)_\RR = 2\div(\face)$ for all $q \in \relint \face$.
\end{definition}

We have found a way to go from faces to divisors. We can also go in the other direction, as follows.

\begin{definition}
    Given a totally real effective divisor $D$, we define \[\face_{D}\coloneqq \{ q \in P \colon \div(q) \ge 2D \}\] the face of $\pos$ consisting of quadratic forms vanishing on $X(\RR)$ with at least twice the multiplicity given by $D$. We call $F_D$ the \emph{face associated to the divisor $D$}.
\end{definition}
\begin{lemma}
\label{lem:div_to_face}
    $\face_D$ is a face of $\pos$.
\end{lemma}
\begin{proof}
    Since $V_D = \{q\in \RR[X] \colon \div(q) \ge 2D \}$ is a linear space, we conclude that $F_D$ is a convex cone contained in $\pos$. To show that it is a face, pick $q_1, q_2 \in \pos$ with $q_1+q_2 \in \face_D$. So we have $\mult_A(q_1+q_2) \ge \mult_A(2D)$ for all $A\in X(\RR)$. Since $q_1$ and $q_2$ are nonnegative on $X(\RR)$, the valuation $\nu_A(q_1+q_2)$ is equal to $\min\{\nu_A(q_1),\nu_A(q_2)\}$ so that we also get $\mult_A(q_1) \ge \mult_A(2D)$ and $\mult_A(q_2) \ge \mult_A(2D)$ for all $A \in X(\RR)$. Therefore $q_1,q_2 \in \face_{D}$, showing that $\face_D$ is a face. 
\end{proof}

Denote by $\faces$ the set of faces of $\pos$ and by $\Div_{\ge 0}(X(\RR))$ the set of totally real effective divisors. We define
\newline
    \begin{minipage}{0.45\textwidth}
        \begin{align*}
            \Phi \colon \faces\setminus\{\, \{0\}\, \} & \longrightarrow \Div_{\ge 0}(X(\RR)) \\
            \face & \longmapsto \div(\face)
        \end{align*}    
    \end{minipage}
    \begin{minipage}{0.45\textwidth}
        \begin{align*}
            \Psi \colon \Div_{\ge 0}(X(\RR))\supset \Im(\Phi)&        \longrightarrow \faces  \\
            D & \longmapsto \face_D
        \end{align*}    
    \end{minipage}\\
$\Phi$ and $\Psi$ are, by definition, two order-reversing maps between the poset of (nonzero) faces $\faces\setminus\{\, \{0\}\, \}$ and the face divisors $\Im(\Phi) \subset \Div_{\ge 0}(X(\RR))$. Here, we order $\faces$ by inclusion and $\Im(\Phi)$ by the usual order on $\Div(X)$ defined as such: $\sum_{P\in X} a_P P\leq \sum_{P\in X} b_P P$ if $a_P\leq b_P$ for all $P\in X$. 

We now show that every face can be realized as the face associated to a totally real effective divisor. In other words, we prove that $\Psi$ is surjective.

\begin{theorem}\label{thm:char_faces}
    Let $X \subset \PP^n$ be a totally real curve and $\pos = \pos_{X, 2}$ the convex cone of nonnegative quadrics. For any face $\face \subset \pos$, $F \neq \{0\}$ the following statements hold.
    \begin{enumerate}
        \item $\face = \face_{\div(\face)} = (\Psi \circ \Phi)(\face)$;
        \item The linear span of $\face$ is $V_{\div(\face)} = \{ \, q \in R_2 \colon \div q \ge 2\div(\face) \, \}$;
        \item $\relint \face = \{ \, q\in \pos \colon (\div q)_{\RR} = 2\div(\face) \, \}$.
    \end{enumerate}
\end{theorem}
\begin{proof}
    Let $\div(\face)$ be the face divisor of $\face$. For any $g \in \relint \face$, we have $(\div g)_\RR = 2\div(\face)$ by definition. We now show that if $(\div g)_\RR = 2\div(\face)$ for some $g\in \pos$ then $g$ lies in the relative interior of $\face$.
    
    Consider the vector space $V = V_{\div(\face)} = \{ \, q \in R_2 \colon (\div q) \ge 2\div(\face) \, \}$. Then we have $\face_{\div(\face)} \subset V$ by definition. We now show that any $g\in P$ with $(\div g)_\RR = 2 \div(\face)$ belongs to the interior of $\face_{\div(\face)}$ in $V$, which in particular implies that $V$ is the linear span of $\face_{\div(F)}$.
    Since $(\div g)_\RR = 2\div(\face)$ we have $g \in \face_{\div(\face)}$. 
    We only need to show that for any $q\in V$ there is an $\epsilon>0$ such that $g+\epsilon q \in \pos$ , i.e.,~that $g+\epsilon q \ge 0$ on $X(\RR)$. This follows from compactness of $X(\RR)$. 
    We switch to the affine cone so that $g$ and $q$ are actual functions on $\widehat X(\RR)\cap \SS^n$.
    For every $A \in \supp \div(\face)$ let $x_A \in \widehat X(\RR)\cap \SS^n$ be an affine representative. Since the multiplicity of $g$ at $A$ is smaller or equal to the multiplicity of $q$ at $A$, and since $g$ is nonnegative, there is an open neighborhood $U_A$ of $x_A$ in $\widehat X(\RR)\cap \SS^n$ and an $\epsilon>0$ such that $g+\epsilon q \ge 0$ on $U_A$. Since $g(-x) = g(x)$ and $g$ vanishes only at finitely many points in $X(\RR)$, there exists an $\epsilon>0$ such that $g+\epsilon q \ge 0$ on $U \coloneqq \bigcup_{A \in \supp \div(\face)} (U_A \cup -U_A)$.
    Next, since all the zeroes of $g$ belong to $U$, we conclude that $g$ is strictly positive on the compact set $K\coloneqq(\widehat{X}(\RR)\cap \SS^n) \setminus U$. Since $q$ is bounded on $K$ and $g$ has a strictly positive minimum of $K$ we get $g+\epsilon q >0$ on $K$ for sufficiently small $\epsilon>0$. Overall, we conclude that there is an $\epsilon>0$ such that $g + \epsilon q$ is nonnegative on $X(\RR)$, showing that $g$ is in the relative interior of $\face$.
    
    This also implies that any interior point of $\face$ is an interior point of $\face_{\div(\face)}$, and thus by \cite[Cor. 18.1.2]{rockafellarConvexAnalysis1970a} we have $\face = \face_{\div(\face)}$. Moreover, the linear span of $\face$ is $V$.

    Above, we have shown the inclusion $\relint \face_{\div(\face)} \supset \{ \, q\in \pos \colon (\div q)_{\RR} = 2\div(\face) \, \}$. We now prove the reverse inclusion as well.
    For $f \in \face_{\div(\face)} \setminus \{ \, q\in \pos \colon (\div f)_{\RR} = 2\div(\face) \, \}$ there exists a real point $A \in \supp \div(\face)$ such that $\mult_A(q) > \mult_A(\div(\face))$. Consider then $L \colon V_A \to \RR$ as in the proof of \Cref{lem::rel_int}. By definition we have $L(f) = 0$ and $L(g) > 0$ for all $ g \in \{ \, q\in \pos \colon (\div q)_{\RR} = 2\div(\face) \, \} \supset \relint \face_{\div(\face)}$, concluding the proof.
\end{proof}

\Cref{thm:char_faces} allows us to describe the duality between faces and face divisors precisely, as follows.

\begin{theorem} \label{cor:galois}
    The functions\newline
    \begin{minipage}{0.45\textwidth}
        \begin{align*}
            \Phi \colon \faces\setminus\{\, \{0\}\, \} & \longrightarrow \Div_{\ge 0}(X(\RR)) \\
            \face & \longmapsto \div(\face)
        \end{align*}    
    \end{minipage}
    \begin{minipage}{0.45\textwidth}
        \begin{align*}
            \Psi \colon \Div_{\ge 0}(X(\RR))\supset \Im(\Phi) &        \longrightarrow \faces  \\
            D & \longmapsto \face_D
        \end{align*}    
    \end{minipage}\\
    form a Galois connection between the positive dimensional faces $\faces\setminus\{\, \{0\}\, \}$, ordered by inclusion, and the set of face divisors $\Im(\Phi)\subset \Div_{\ge 0}(X(\RR))$. This means that, for all $F \in \faces$ and $D \in \Im \Phi$ we have
    \[
        \div(\face) \le D \iff \face \supseteq \face_{D}.
    \]
\end{theorem}
\begin{proof}
    Let $\face$ and $D$ be such that $\div(\face) \le D$. By definition this implies $\face_{\div(\face)} \supset \face_{\div(\face_D)}$. From \Cref{lem:div_to_face} we deduce that $\face_D$ is a face of $\pos$, while from \Cref{thm:char_faces} $\face_D = \face_{\div(\face_D)}$ and $\face = \face_{\div(\face)}$. Therefore $\face = \face_{\div(\face)} \supset \face_{\div(\face_D)} =\face_D$, as we wanted.

    For the converse implication, let $D = \div(\widetilde \face) \in \Im \Phi$ and $\face \supset \face_{D} = \face_{\div(\widetilde \face)}$. This implies that $\div(\face) \le \div(\face_{\div(\widetilde \face_1)}) = \div(\widetilde \face) = D$, concluding the proof.
\end{proof}

\Cref{thm:char_faces} implies that $\Psi$ is surjective. On the other hand, $\Phi$ is not. 
Characterizing the image of $\Phi$ is difficult, and it is naturally connected to existence of nonnegative quadrics, as we now show.

\begin{lemma}\label{lem:image}
    Let $D$ be a totally real effective divisor. The following are equivalent:
    \begin{enumerate}
        \item $D$ is a face divisor, i.e., $D \in \Im(\Phi)$;
        \item $D = \div(\face_{D})$;
        \item There exists $q \in \pos$ such that $(\div q)_\RR = 2D$.
    \end{enumerate}
\end{lemma}
\begin{proof}
    For $(i) \implies (ii)$, we have to prove $\div(\face) = \div(\face_{\div(\face)})$, which follows from \Cref{thm:char_faces}.
    The converse implication $(ii) \implies (i)$ is trivial.
    
    $(ii) \implies (iii)$ follows by taking $q\in \relint \face_D$ because we have $(\div q)_\RR = 2 \div(\face_D) = 2D$ for such a $q$.
    The final implication $(iii) \implies (ii)$ follows from the fact that $(\div q)_\RR = 2D$ implies $q \in \relint \face_{D}$, and thus $D = \div(\face_D)$.
\end{proof}

Studying the set $\Im \Phi$ is challenging. The following example, for instance, shows that $\Im (\Phi)$ is in general not a lower set in $\Div_{\ge 0}(X(\RR))$, i.e.  for $D, D' \in \Div_{\ge 0}(X(\RR))$, $D \in \Im \Phi$ and $D' \le D \centernot\implies D' \in \Im \Phi$.

\begin{example} \label{ex:nolowerset}
Consider the projective closure $X\subset \PP^2$ of the affine plane quartic defined by \[2(1-x_1^2-2x_2^2)(1-2x_1^2-x_2^2) - (4x_1^2-1)(4x_2^2-1)=0\]
Its real locus consists of four disjoint ovals, symmetric with respect to the $x_1$ and $x_2$ axis. Setting $a=\sqrt{1-\sqrt{112}/14}$, we can verify that $X(\RR)$ intersects the quadric $q=a^2-x_1^2$ in the four points $(\pm a, \pm \sqrt{1-3\sqrt{7}/14})$ with multiplicity $2$, and that $q$ is nonnegative on $X(\RR)$. If we denote these four points by $A_1, \dots , A_4$ and we define $D = A_1 + \dots + A_4$, we therefore have that $\div q = (\div q)_\RR = 2D$, and thus $D \in \Im \Phi$.

We now set $D' = A_1+A_2+A_3$. To show that $\Im \Phi$ is not a lower set in the totally real effective divisors, we prove that $D' \notin \Im \Phi$. Equivalently, we want to show that there does not exist $q' \in \pos$ s.t. $(\div q')_\RR = 2D'$.
Let $q' \in R_2$ be a quadric such that $\div q' \ge 2D$. It follows from the Cayley-Bacharach theorem (e.g. \cite[Th.~CB7]{eisenbudCayleyBacharachTheoremsConjectures1996}, applied to the hypersurfaces $X$ and $q=0$ with $\Gamma'=2D'$ and $\Gamma''=2P_4$) that $q'$ vanishes with multiplicity $2$ at $P_4$, which implies that $\div q' \ge (\div q')_{\RR} \ge 2D$. Therefore, every nonnegative quadric vanishing at $P_1, P_2, P_3$ vanishes also at $P_4$, showing that $D' = P_1+P_2+P_3 \notin \Im \Phi$. Thus $\Im \Phi$ is not a lower set.
\end{example}

\subsection{Dimension of faces}\label{sec:dimension_normality}
In this section, we use the Riemann-Roch theorem (see e.g. \cite{shafarevichBasicAlgebraicGeometry2013}) to compute dimensions of faces of the cone of nonnegative quadrics on $X$. We denote $ \mathcal{L}(D) = \left\{ \, f\in \CC(X) \mid D + \div(f) \ge 0 \, \right\} $ the Riemann-Roch spaces, and
we write $\ell(D)$ for $\dim_\CC(\mathcal{L}(D))$. The Riemann-Roch Theorem states $\ell(D) = \deg(D) + 1 - g + \ell(K-D)$, where $g$ is the genus of $X$ and $K$ is a \emph{canonical divisor} of $X$.
For this computation to be reflected in the chosen embedding $X\subset \PP^n$ of our curve, we assume that the curve is \emph{projectively normal}, which means that the linear systems $|dH|$ for any $d\in \NN$ are complete, where $H$ is a hyperplane in $\PP^n$. Concretely this means that, for any divisor $D\in \Div(X)$ that is linearly equivalent to the zero divisor $\div g$ of a form $g \in \CC[X]_d$, there exists a form $h\in \CC[X]_d$ with $D = \div h$.

We can now investigate, in the projectively normal case, the dimension of faces $\face \in \faces$. 

\begin{lemma}\label{lem:dimension}\label{lem:separating_quadric}
    Let $X\subset \PP^n$ be a totally real curve, and let $\{ 0 \}\neq\face \subset \pos$ be a nonzero face.
    If $X\subset \PP^n$ is projectively normal and $q \in \relint \face$, then $\dim \face = \ell((\div q)_{\CC})$.
\end{lemma}
\begin{proof}
    Since $D = \div(\face) \in \Im \Phi$ we have $D = (\div q)_{\RR}$ for any $q \in \relint \face$ by \Cref{thm:char_faces}. As in \Cref{thm:char_faces} we write $V_D = \{ \, p \in R_2 \colon \div p \ge 2D \, \}$. For any $q \in \relint \face$ consider the map:
    \[
        V_D \longrightarrow \mathcal{L}((\div q)_{\CC}), \quad f \longmapsto f/q.
    \]
    This is clearly well-defined and injective, and it is surjective since $X \subset \PP^n$ is projectively normal. Therefore $\dim V_D = \ell((\div q)_{\CC})$. By \Cref{thm:char_faces} we have $\dim \face = \ell((\div q)_{\CC})$.
\end{proof}

Thanks to \Cref{lem:dimension} and \Cref{thm:char_faces}, when we consider a totally real projectively normal curve $X \subset \PP^n$, we only need to consider nonnegative quadrics with a large number of real zeroes, in order to understand the dimension of all the small dimensional faces $\face \subset \pos$. We make this statement precise in the following proposition for the extreme rays, i.e. for faces of dimension one.

\begin{proposition}\label{prop:max_zeroes}
    Let $X\subset \PP^n$ be a totally real projectively normal curve of genus $g$. 
    If $\RR_{\ge 0} \cdot q$ is an extreme ray of $\pos$, then $\deg ((\div q)_{\RR}) \ge 2\deg X -g$. In particular, if all the real zeroes of $q$ have multiplicity two, then $q$ has at least $\deg X - g/2$ real zeroes.
\end{proposition}
\begin{proof}
    If $\RR_{\ge 0} \cdot q$ is an extreme ray of $\pos$, then combining \Cref{lem:dimension} and Riemann's inequality we have:
    \[
        1 = \ell((\div q)_{\CC}) \ge \deg((\div q)_{\CC}) - g + 1
    \]
    This implies that:
    \[
        2 \deg X - \deg ((\div q)_{\RR}) = \deg ((\div q)_{\CC}) \le g 
    \]
    concluding the proof.
\end{proof}

\section{Extreme rays of the nonnegative cone}
\label{sec:extreme_rays}

We study the \emph{extreme rays} of the cone $P$ of nonnegative quadrics, which we denote by $\cE(P)$. 
The complete description of the extreme rays of $\pos$ is difficult for curves of arbitrary genus. In the following, we will mostly focus on real rooted extreme rays, i.e. those that they are defined by nonnegative quadrics with only real zeroes.  We use
\[
    \cE^{\mathrm{rr}}(\pos) \coloneqq \{ \, \RR_{\ge 0} \cdot q \in \cE(\pos) \mid \div q = (\div q)_{\RR} \, \}
\]
to denote the set of such extreme rays. \Cref{ex:extreme_no_max_zeroes} below shows that, in general, $\cE^{\mathrm{rr}}(\pos) \subsetneq \cE(\pos)$. However, for genus one curves, \Cref{prop:max_zeroes} implies that the extreme rays are all real rooted. We study them further in \Cref{sec:genus_one,sec:cara,sec:dual_plane_cubics}. 

\begin{example}\label{ex:extreme_no_max_zeroes}
    Let $X(\RR) \subset \PP^2(\RR)$ be the plane sextic defined by $h=x_1^6+x_2^6-x_0^6=0$ and let $q = x_1^2+x_2^2 - x_0^2\in \RR[x_0, x_1, x_2]_2 \cong \RR[X]_2$. The curve $X$ is a smooth plane curve of degree $6$ and hence has genus $g=10$ from the genus-degree formula. \Cref{prop:max_zeroes} implies that any nonnegative quadric defining an extreme ray of $\pos$ has between $2 \deg X - g = 2$ and $12$ real zeroes, counted with multiplicity.
    
    The curves $h$ and $q$ intersect with multiplicity two in the real points $A_1=(1:0:1)$, $A_2=(1:0:-1)$, $A_3=(1:1:0)$ and $A_4=(1:-1:0)$, and with multiplicity two in the complex conjugate points $A_5=(0:1:i)$ and $A_6=\sigma(A_5)=(0:1:-i)$. In other words, $(\div q)_\RR = 2(A_1+\dots + A_4)$ and $(\div q)_\CC = 2(A_5+\sigma(A_5))$.
    
    The curve $X(\RR)$ is connected, and $q$ is nonnegative on $X(\RR)$.
    By \Cref{prop:max_zeroes}, to show that $\RR_{\ge 0} \cdot q$ is an extreme ray of $\pos$ it is sufficient to verify that $\ell((\div q)_\CC) = \ell(2(A_5+\sigma(A_5))=1$. This can be done using the Cayley-Bacharach theorem, see e.g. \cite[Th.~ CB7]{eisenbudCayleyBacharachTheoremsConjectures1996} or using any computer algebra system.  Alternatively, one can directly verify that the double vanishing at $A_1, \dots , A_4$ imposes enough conditions to obtain a one-dimensional family of solutions.
    % Here is the sagemath code used to verify the computation. the notation is z=x_0, x=x_1, y=x_2.
    %R.<z> = FunctionField(QQbar)
    %S.<Y> = R[]
    %L.<y> = R.extension(Y^6 - z^6 + 1)
    %D = L.maximal_order().ideal(y-I,z).divisor()
    %E = L.maximal_order().ideal(y+I,z).divisor()
    %B = (2*(D+E)).basis_function_space()
    %print(len(B))
    %print(B)

    Notice in particular that \[ 2\deg X - g = 2 < \deg(\div q)_\RR = 8<2\deg X = 12\]
    showing that $\RR_{\ge 0} \cdot q$ is an extreme ray defined by a quadric with a nonmaximal number of real zeroes, and not attaining the lower bound of the inequality in \Cref{prop:max_zeroes}.
\end{example}

\subsection{Extreme rays and the Jacobian}
\label{sec:double}
We propose a general framework to study $\cE(\pos)$, based on the properties of the Jacobian of a real curve $X$, for which we refer the reader to \cite{grossRealAlgebraicCurves1981a, vinnikovSelfadjointDeterminantalRepresentations1993,baldi2025totallyrealdivisorscurves}.

Let $(X, \sigma)$ be a smooth irreducible real curve of genus $g$. For every $f \in \CC(X)\setminus \{\, 0 \, \}$, we can consider the divisor $\div f$ of poles and zeroes. Such divisors form a normal subgroup of $\Div(X)$, and the quotient $\Cl{(X)}$ is called the \emph{divisor class group}. 
So two divisors $D$ and $D'$ represent the same class $[D] = [D'] \in \Cl{(X)}$ if and only if there exists $f \in \CC(X)\setminus\{0\}$ such that $D = D' + \div f$.
We denote $\Cl^d(X)$ the subgroup of divisor classes of degree $d$. $\Cl^0{(X)}$ is isomorphic to the Jacobian variety $\jac = \jac(X) \cong \CC^g/\Lambda$, which is a (complex) $g$-torus. There is a natural action of $\sigma$ on $\jac$,
which gives $\jac$ the structure of $\RR$-variety.
As usual, we denote by $\jac (\RR)$ the real part of $J$, i.e. the fixed points of $\sigma$. Concretely, a divisor class $[D]$ is in $\Cl^0(X)(\RR) \cong \jac(\RR)$ if the divisor $\sigma(D)$ is linearly equivalent to $D$, which implies $[\sigma(D)] = [D]$. 

We define a map $\cE(\pos) \to \jac(\RR)$. Since $X$ is smooth and totally real, we have $X(\RR) = S_1 \sqcup \dots \sqcup S_r$, where $S_i$ is diffeomorphic to a circle for all $i$, and $0<r\le g+1$ by Harnack's inequality. Furthermore, there exist smooth, simple loops
$R_1, \dots , R_m\subset X$ such that deleting these loops and the real locus $S_1, \dots, S_r$ from the Riemann surface $X(\CC)$ disconnects it: $X(\CC) \setminus (\bigcup_i S_i \cup \bigcup_j R_j) = Y_1 \sqcup Y_2$ is disconnected, with $\sigma(R_j) = R_j$, $\sigma(Y_1) = Y_2$ and $\sigma(Y_2) = Y_1$. We refer the reader to \cite[Sec.~2]{baldi2025totallyrealdivisorscurves} for more details. The action of $\sigma$ on $R_j$ has no fixed points, and we can write $R_j = T_{j,1} \sqcup T_{j,2}$, with $\sigma(T_{j,1}) = T_{j,2}$, $\sigma(T_{j,2}) = T_{j,1}$ and each $T_{j,k}$ is diffeomorphic to $[0,1)$. These choices allow us to split divisors $D$ supported outside $X(\RR)$ into two parts. If $\supp(D) \cap X(\RR) = \emptyset$, write $\widetilde{D}$ for the unique divisor supported on $Y_1 \cup T_{1,1} \cup \dots \cup T_{m,1}$ such that $D = \widetilde{D} + \sigma(\widetilde{D})$. If $H$ denotes the hyperplane divisor, we can then define the map
\begin{align*}
    \mathfrak{D} \colon \cE(\pos) & \longrightarrow \jac \\
    \RR_{\ge} \cdot q & \longmapsto [\frac{1}{2} (\div(q))_{\RR} + \widetilde{(\div q)_{\CC}} - H] \eqqcolon \mathfrak{D}(q)
\end{align*}
Notice that the map is well-defined, since the degree of the divisor $\frac{1}{2} (\div(q))_{\RR} + \widetilde{(\div q)_{\CC}} - H$ is zero, and that $\mathfrak{D}(q) + \sigma(\mathfrak{D}(q)) = 0$ by definition. The operation $\widetilde{\cdot}$, and hence $\mathfrak{D}$, depend on the (non-unique) choice of $R_1, \dots , R_m$. A related construction, independent of such a choice but not refined enough for our purposes, has been studied in \cite[Sec.~2]{scheidererSumsSquaresRegular2000}. Notice that, if $\RR_{\ge 0} \cdot q \in \cE^\rr(\pos)$, then $\double(q) =  [\frac{1}{2} (\div(q))_{\RR} - H]$.

\subsection{Real rooted extreme rays} 
\label{sec:real_rooted}
A full understanding of $\cE(\pos)$ will require studying the image $\mathfrak{D}(\cE(\pos))$ in detail. We begin by investigating $\mathfrak{D}(\cE^{\mathrm{rr}}(\pos))$. For this, we are interested in the $2$-torsion points of $\jac(\RR)$.
\begin{definition}\label{def:two_torsion}
    We denote $\jac_2 = \{ \, \alpha \in \jac \colon 2\alpha =0 \, \}$ the subgroup of $\jac$ consisting of $2$-torsion points, and we denote the \emph{real $2$-torsion points} by $\jac(\RR)_2 \coloneqq \jac_2 \cap \jac(\RR)$.
\end{definition}

It is known that the number of real $2$-torsion points depends only on the genus $g$ of $X$ and the number of connected components $r$ of $X(\RR)$.
\begin{lemma}
    [{\cite[§~5]{grossRealAlgebraicCurves1981a}}] \label{lem:torsion_genus_components}
    If $(X,\sigma)$ is a totally real curve of genus $g$ such that $X(\RR)$ has $r>0$ connected components, then $\jac(\RR)_2 \cong (\ZZ/2\ZZ)^{g+r-1}$.
\end{lemma}

Among all real $2$-torsion points, we are interested in those defined by nonnegative rational functions.
\begin{lemma}\label{lem:positive_torsion}
    Let $\alpha \in \jac(\RR)_2$ and let $D, D'$ be divisors representing $\alpha$. If $f, f' \in \RR(X)$ are such that $\div f =2D$ and $\div f' = 2D'$, then there exists $g \in \RR(X)$ such that $f' = f g^2$. In particular, if $f \ge 0$ on $X(\RR)$ (wherever it is defined), then $f' \ge 0$ on $X(\RR)$ (wherever it is defined).
\end{lemma}
\begin{proof}
    Let $D, D',f, f'$ and $\alpha$ be as in the statement. As $[D] = \alpha = [D']$, there exists $g\in \RR(X)$ such that $D+\div g = D'$.
    Hence $\div f' = 2D' = 2(D + \div g) = \div f g^2$. Therefore, up to a multiplicative constant, $f' = f g^2$.
\end{proof}
\begin{definition}\label{def:positive_torsion}
    We say that $\alpha \in \jac(\RR)_2$ is \emph{positive} if, for any divisor $D$ representing $\alpha$, there exist $f\in \RR(X)$ such that $f \ge 0$ on $X(\RR)$ (wherever it is defined) and $\div f = 2D$ (see \Cref{lem:positive_torsion}). We denote the subgroup of positive $2$-torsion points by $\jac(\RR)_2^+$.
\end{definition}

{We show that} the number of positive $2$-torsion points $\jac(\RR)_2^+$ depends only on the genus of $X$, and not on the topology of $X(\RR)$: {see {\cite[§~5]{geyerUeberlagerungenBerandeterKleinscher1977}} for a related discussion.}

\begin{lemma}\label{lem:positive_torsion_genus}
    If $(X,\sigma)$ is a totally real curve of genus $g$, then $\jac(\RR)_2^+ \cong (\ZZ/2\ZZ)^{g}$.
\end{lemma}
\begin{proof}
    Write $X(\RR) = Y_1 \sqcup \dots \sqcup Y_{r}$ as the disjoint union of its connected components. For all $\alpha = [D] \in \jac(X)_2$, we can choose $f_\alpha \in \RR(X)$ such that $\div f = 2D$ and $f \ge 0$ on $Y_1$. The sign of any such $f_\alpha$ on $Y_2, \dots , Y_{r}$ only depends on $\alpha$ because for any two rational functions $f$ and $f'$ such that $\div(f) = 2D$ and $\div(f') = 2D'$ and $[D] = [D'] = \alpha \in \jac(X)$ differ by a square, see the proof of \Cref{lem:positive_torsion}. We can therefore define the map $\jac(X)_2 \to \{ \, \pm 1 \, \}^{r-1}$, which associates to every $\alpha$ the tuple of signs of $f_\alpha$ on $Y_2, \dots, Y_r$. This is a well-defined group morphism, whose kernel is equal to $\jac(\RR)_2^+$. The morphism is surjective (see \cite[§~5]{geyerUeberlagerungenBerandeterKleinscher1977} or \cite{kummer2026nonnegativepolynomialsgeneralizedellipticv3}), and thus \Cref{lem:torsion_genus_components} implies $\jac(\RR)_2^+ \cong (\ZZ/2\ZZ)^{g}$.
\end{proof}

We are now ready to show how positive $2$-torsion points can be used to characterize totally real divisors that arise as intersections of nonnegative quadrics on $X\subset \PP^n$ with the maximal number of real zeroes.

\begin{proposition}\label{prop:2torsion}
    Let $X \subset \PP^n$ be a totally real smooth irreducible projectively normal curve. Let $[H]$ be the hyperplane divisor class, and let $D \in \Div_{\ge 0}(X(\RR))$ be an effective totally real divisor such that $\deg D = \deg X$ and $[2D] = [2{H}]$. Then the following are equivalent.
    \begin{enumerate}
        \item There exists a quadric $q\in R_2$ such that $2D = (\div q) = (\div q)_\RR$.
        \item $[D-{H}] \in \jac(\RR)$ is a $2$-torsion point.
    \end{enumerate}
    Moreover, $q$ is nonnegative if and only if $[D-{H}] \in \jac(\RR)_2^+$.
\end{proposition}
\begin{proof}
    For $(i) \implies (ii)$, consider the rational function $q/\ell^2$, where $\ell$ is {a} linear form. Then $\div (q/\ell^2) = 2D-2{\div \ell}$.
    If $q \in \pos$, then $q/\ell^2$ is nonnegative on $X(\RR)$ and thus $[D-{H}] {= [D - \div \ell]}\in \jac(\RR)_2^+$.
        
    For $(ii) \implies (i)$, as $[D-{H}]\in \jac(\RR)_2$ there exists $f \in \RR(X)$ and $E \in \abs{E}$ such that $\div f = 2D-2{E}$. Therefore $2D = \div f\ell^2$ {for some linear form $\ell$}, and by projective normality there exists $q \in R_2$ such that $2D = \div q = (\div q)_{\RR}$.
    Furthermore, if $[D-{H}]\in \jac(\RR)_2^+$ then $f$ is nonnegative on $X(\RR)$ and thus $q\in \pos$.
\end{proof}

\begin{theorem}\label{thm:maxzeroesfaces}
    Let $X \subset \PP^n$ be a totally real smooth irreducible projectively normal curve. Let $[H]$ be the hyperplane divisor class, and let $D \in \Div_{\ge 0}(X(\RR))$ be a divisor of degree $\deg D = \deg X$ with $[2D] = [2{H}]$.
    Then the following are equivalent:
    \begin{enumerate}
        \item $\face_D \subset \pos$ is an extreme ray of $\pos$;
        \item $[D - {H}]$ is a {positive} $2$-torsion point of the Jacobian.
    \end{enumerate}
\end{theorem}
\begin{proof}
    Combine \Cref{prop:2torsion} and \Cref{thm:char_faces}.
\end{proof}

Therefore, the existence of a nonnegative quadric with a maximal number of real zeroes can be read from the positive $2$-torsion points in the Jacobian. However, to be able to apply \Cref{lem:positive_torsion_genus} to count their number, we need to be sure that the torsion points are totally real, i.e. that their divisor classes can be represented using only points of $X(\RR)$.

\begin{definition}\label{def:NX}
    Let $X$ be an real curve. We say that a real divisor class $\beta \in \Cl^k(X)$ is \emph{totally real} if there exists $A_1, \dots , A_k\in X(\RR)$ such that $\beta = [A_1+\dots +A_k]$. We denote by $N(X)$ the smallest $k$ such that all real divisor classes in $\Cl^k{X}$ are totally real.
\end{definition}

The study of the invariant $N(X)$, called \emph{totally real divisor threshold}, was initiated by Scheiderer in \cite{scheidererSumsSquaresRegular2000}; \cite[Th.~2.7]{scheidererSumsSquaresRegular2000} shows that $N(X)$ is finite for every real curve $X$. The papers \cite{huismanGeometryAlgebraicCurves2001, monnierDivisorsRealCurves2003, baldi2025totallyrealdivisorscurves}) give bounds on $N(X)$ depending on the topological type of the curve and on its intrinsic metric properties. The fact that $N(X)<+\infty$ allows us to describe $\cE^\rr(\pos_{X, 2d})$ in sufficiently high degree $2d$.

\begin{lemma} \label{lem:totally_real_torsion}
    Let $X \subset \PP^n$ be a totally real smooth irreducible projectively normal curve, {and denote $[H]$ the hyperplane divisor class}.
    For every $d\in \NN$, and for every positive $2$-torsion point $\alpha \in \jac(\RR)_2^+$ such that
    $\alpha+[d{H}]$ is totally real, there exist a real rooted extreme ray $ \RR_{\ge 0} \cdot q_\alpha \in \cE^\rr(\pos_{X, 2d})$ such that $\double(q_\alpha) = \alpha$.
\end{lemma}
\begin{proof}
    Since $\alpha$ is a positive $2$-torsion point, there exists $f \in \RR(X)$ nonnegative on $X(\RR)$ such that $\div f = 2D$ and $\alpha =[D]$. As $\alpha + [d{H}] = [D+d{H}]$ is totally real by hypothesis, there exists a totally real divisor $D''$ such that $[D+d{H}] = [D'']$. Moreover, since the embedding is projectively normal and $[2D''] = [2d{H}]$, there exists $q_\alpha \in R_2$ such that $\div q_\alpha = (\div q_\alpha)_\RR = 2D''$. Finally, since {$\alpha \in \jac (\RR)_2^+$}, we have (up to a global sing change) $q_\alpha \ge 0$ on $X(\RR)$, i.e. $q_\alpha \in \pos_{X, 2d}$. From \Cref{thm:maxzeroesfaces} we conclude that $\RR_{\ge 0} \cdot q_{\alpha}$ is an extreme ray of $\pos_{X, 2d} \cong \pos_{\nu_{n,d}(X), 2}$ and, by construction, we have $\double(q_\alpha) = \alpha$.
\end{proof}

\begin{corollary}\label{cor:families_torsion}
    Let $X\subset\PP^n$ be a totally real smooth irreducible curve of genus $g$. For all sufficiently large $d$, $\mathfrak{D}(\cE^{\mathrm{rr}}(\pos_{X, 2d})) = \jac(\RR)_2^+ \cong (\ZZ/2\ZZ)^{g}$. Moreover, $f_1, f_2$ defining real rooted extreme rays of $\pos_{X, 2d}$ satisfy $\mathfrak{D}(f_1) = \mathfrak{D}(f_2)$ if and only if there exists $g \in \RR(X)$ such that $f_1 = g^2 f_2$.
\end{corollary}
\begin{proof}
    Recall that $\pos_{X, 2d} \cong \pos_{\nu_{n,d}(X), 2}$, {and if $[H]$ is the hyperplane class on $X$, then $[dH]$ is the hyperplane class of $\nu_{n,d}(X)$}. Moreover, if $d$ is big enough, $\nu_{n,d}(X)$ is projectively normal and
    \[
        d \cdot \deg X = \deg \nu_{n,d}(X) \ge N(X),
    \]
    see \Cref{def:NX}. This implies that for every $\alpha \in \jac(\RR)_2^+$, the divisor class $\alpha + [d{H}]$ is totally real. Therefore, from \Cref{lem:totally_real_torsion} we deduce that for all $\alpha \in \jac(\RR)_2^+$ there exist $q_{\alpha} \in \pos_{X, 2d}$ real rooted and generating an extreme rays such that $\double(q_\alpha)  =\alpha$. We recall from \Cref{lem:positive_torsion_genus} that $\jac(\RR)_2^+ \cong (\ZZ/2\ZZ)^{g}$, and conclude by applying \Cref{lem:positive_torsion} that $\mathfrak{D}(f_1) = \mathfrak{D}(f_2)$ if and only if there exists $g \in \RR(X)$ such that $f_1 = g^2 f_2$.
\end{proof}

\Cref{cor:families_torsion} highlights the difference between sums of squares and nonnegative forms. Indeed, if the sums of squares and nonnegative cones are equal, then they have the same extreme rays. But all sums of squares extreme rays correspond only to the zero of the Jacobian, which is a positive $2$-torsion point. If the genus is at least one, then for all sufficiently large $d$ there are positive $2$-torsion points which correspond to non-sums of squares extreme rays, and thus we have $\Sigma_{X, 2d} \subsetneq \pos_{X, 2d}$. This is a particular case of \cite[Th.~1.1]{blekhermanSumsSquaresVarieties2016a}. For a more detailed description in the genus one case, see \Cref{thm:elliptic_normal}.

\section{Nonnegative forms on genus one curves}\label{sec:genus_one}
\subsection{Elliptic normal curves}\label{sec:elliptic_normal}
We now consider in detail curves of genus $1$.  The case of plane cubic curves, which inspired our work, is detailed in elementary terms in \Cref{prop:psd_plane_cubic}.
\begin{definition}
    An \emph{elliptic normal curve} is a non-degenerate, smooth projective curve $X \subset \PP^n$ of genus one such that $\deg X = n+1$.
\end{definition}
Recall that every elliptic normal curve is automatically projectively normal, see e.g. \cite[Ex.~IV.4.2]{hartshorneAlgebraicGeometry1977}. Using this fact it is possible to show that the Veronese embedding of an elliptic normal curve is again an elliptic normal curve: therefore we can restrict to the case of quadratic forms, as we did in \Cref{sec:psd_forms_curves}.
If the genus one curve $X$ is totally real, then $X(\RR)$ has either one or two connected components, and the invariant $N(X)$, introduced in \Cref{def:NX}, is always equal to one \cite{huismanGeometryAlgebraicCurves2001, monnierDivisorsRealCurves2003}.

We use these facts to fully describe the face lattice of $\pos_{X,2}$ for an alliptic normal curve $X$.

\begin{theorem}\label{thm:elliptic_normal}
    Let $X \subset \PP^n$ be a totally real elliptic normal curve, {and denote $[H]$ the hyperplane divisor class}. If we denote $\{ \,O, T \, \} = \jac(\RR)_2^+$ the positive $2$-torsion points, then all the proper faces of $\pos_{X, 2}$ can be described as follows:
    \begin{enumerate}
        \item $\face_D$, where $D$ is a totally real effective divisor of degree $1 \le \deg D \le n$. In this case, we have $\dim \face_D = 2(n+1 - \deg D)$.
        \item $\face_D$, where $D$ is a totally real effective divisor of degree $n+1$ such that ${[D-H]} = O\in \jac(\RR)_2^+$. In this case, $\face_D = \RR_{\ge 0} \cdot q$ is an extreme ray of both $\Sigma=\Sigma_{X, 2}$ and $\pos$.
        \item $\face_D$, where $D$ is a totally real effective divisor of degree $n+1$ such that ${[D-H]} = T\in \jac(\RR)_2^+$. In this case, $\face_D = \RR_{\ge 0} \cdot q$ is an extreme ray of $\pos$, but it is not in $\Sigma=\Sigma_{X, 2}$.
    \end{enumerate}
    In particular, we have $\cE(\pos_{X,2}) = \cE^\rr(\pos_{X,2})$.
\end{theorem}
\begin{proof}
    From \Cref{thm:char_faces}, all the faces of $\pos$ are of the form $\face_D$ for some totally real effective divisor $D$. If $\deg D > \deg X = n+1$, then $\face_D = \{ \, 0 \, \}$.

    From \Cref{thm:maxzeroesfaces}, if $\deg D = n+1$ then $\face_D$ is an extreme ray if and only if $[D-{H}]$ is a positive $2$-torsion point of the Jacobian. Recall that $N(X)=1$ from \cite{huismanGeometryAlgebraicCurves2001, monnierDivisorsRealCurves2003}, since $X(\RR)$ has either $1=g$ or $2 = g+1$ connected components. Moreover $X$ is projectively normal, hence every positive $2$-torsion point determines a family of extreme rays of $\pos$, see  \Cref{cor:families_torsion}.
    
    If $[D-{H}] = O\in \jac(\RR)_2^+$, then there exist $f \in \RR(X)$ and $E \in \abs{H}$ such that $D = {E} + \div f$, and thus by projective normality there exists $\ell \in R_1 = \RR[X]_1$ such that $\div \ell = D$. Therefore $\div \ell^2 = 2D$ and $\face_D = \RR_{\ge 0} \cdot \ell^2$ is an extreme ray of both $\Sigma$ and $\pos$. If $[D-{H}] = T\in \jac(\RR)_2^+$, then $D$ is not the divisor of a linear form, and thus the extreme ray $\face_{D}$ is not generated by a square. This implies that $\face_D$ is not contained in $\Sigma$.

    We now study higher-dimensional faces. We first show that, if $D$ is a totally real effective divisor of degree $1 \le \deg D \le n$, then $D \in \Im (\Phi)$, i.e. there exists $q \in \pos$ such that $(\div q)_\RR = D$ (see \Cref{lem:image}). Assume that $\deg D = n = \deg X -1$. Using the isomorphism between $X$ and its Jacobian, we can always find $A_1\neq A_2 \in X(\RR)$ such that $D+A_1, D+A_2 \in \jac(\RR)_2^+$. From \Cref{thm:maxzeroesfaces}, there exists $q_1, q_2 \in \pos$ such that $\div q_i = (\div q_i)_\RR = D+A_i$. Therefore $q_1 + q_2$ is nonnegative and vanishes exactly on $D$, i.e. $(\div (q_1+q_2))_\RR = 2D$, showing that $D \in \Im \Phi$. A similar argument works for all totally real divisors of degree smaller than $n$.
    
    The dimension of $\face_D$ can then be deduced using \Cref{lem:dimension} and the Riemann-Roch theorem, as follows. Since $D \in \Im \Phi$, then \Cref{lem:image,lem::rel_int} imply that for all $q\in \relint \face_D$ we have $(\div q)_\RR = 2D$. It follows from \Cref{lem:dimension} that $\dim \face_D = \ell((\div q)_\CC)$ for all $q \in \relint \face_D$. Using the Riemann-Roch theorem and the fact that $1 \le \deg (\div q)_\CC = 2(n+1 - \deg D)$, we have:
    \[
        \ell((\div q)_\CC) = \deg (\div q)_\CC + \ell(K - (\div q)_\CC) = 2(n+1 - \deg D)
    \]
    concluding the proof.
\end{proof}

\Cref{thm:elliptic_normal} completely describes the facial structure of the cone of nonnegative quadrics on an elliptic normal curve. Using the fact that the Veronese embedding of an elliptic normal curve is an elliptic normal curve, \Cref{thm:elliptic_normal} also describes all faces of $\pos_{X, 2d} \cong \pos_{\nu_{n,d}(X), 2}$. This is the first complete description of the facial structure of the nonnegative cone beyond Hilbert's cases \cite{schulzeConesLocallyNonNegative2021}. It is therefore the first non-rational example where such a description is shown.

We also remark that we do not explicitly see the distinction between the two possible topologies of $X(\RR)$ in \Cref{thm:elliptic_normal}, as the theorem applies to both the connected and disconnected case. In particular, we are not using the real $2$-torsion points of the Jacobian which are not positive. Such points exist when $X(\RR)$ is disconnected, see \Cref{lem:torsion_genus_components,lem:positive_torsion_genus}, and they will play a key role in \Cref{sec:cara} for the study of Carath\'eodory number of the of the dual cone $\pos^\vee$. 

\subsection{Auxiliary lemmas}
We conclude the section presenting two useful lemmas about the signature signs of quadrics on connected components of elliptic normal curves, that will be needed in \Cref{sec:cara}.

\begin{lemma}\label{lem:positive_oval}
    Let $X \subset \PP^n$ be a totally real elliptic normal curve whose real locus has two connected components. If $Y \subset X(\RR)$ is a connected component, then there exists $q\in\pos_{Y,2} \setminus \pos_{X,2}$ (i.e. $q$ is nonnegative on $Y$ but changes sign on $X(\RR)$) such that $\div q = (\div q)_{\RR} = 2D$, with $\supp(D)$ consist of $\deg X = n+1$ distinct points of $Y$.
\end{lemma}
\begin{proof}
    Fix $A_0 \in X(\RR)$. Let $Y \subset X(\RR)$ be a connected component, and let $A\in Y$ be the (unique) point such that $[A-A_0] = \alpha \in \jac(\RR)_2 \setminus \jac(\RR)_2^+$ (such a point exists since genus one curves are isomorphic to their Jacobian). Therefore, it follows from \cite[Problem~5.41]{fultonAlgebraicCurves1989} that there exists $q_A \in R_2$ such that $\div q_A = 2 (n+1) A$. As $[A-A_0]$ is not a positive $2$-torsion point, $q_A$ changes sign on $X(\RR)$, and we may assume that $q_A \in \pos_{Y,2} \setminus \pos_{X,2}$. Moving one pair of points at a time along $Y$, it is possible to continuously perturb the $(n+1)$-tuple $(A,\dots, A)$ to $(A_1, \dots , A_{n+1}) \in Y^{n+1}$, in such a way the $A_i$'s are pairwise distinct and $A_1 \oplus \dots \oplus A_{n+1} = A$.
    Therefore there exists $q \in \pos_{Y,2} \setminus \pos_{X,2}$, perturbation of $q_A$, such that $\div q = 2 (A_1 + \dots + A_{n+1})$, concluding the proof.
\end{proof}

\begin{lemma}\label{lem:changing_sign}
    Let $X \subset \PP^n$ be a totally real elliptic normal curve whose real locus has two connected components. If $X(\RR) = Y_1 \sqcup Y_2$ and $B_1, \dots , B_n \in Y_1$, then there exists $q \in \R_2$ such that:
    \begin{enumerate}
        \item $(\div q)_{\RR} = 2(B_1 + \dots + B_n)$;
        \item $q$ is nonnegative on $Y_1$;
        \item $q$ is strictly negative on $Y_2$.
    \end{enumerate}
\end{lemma}
\begin{proof}
    Consider the non-positive $2$-torsion points $\{ \,T_2, T_3 \, \} = \jac(\RR)_2 \setminus \jac(\RR)_2^+$, and define $A_i = T_i \ominus B_1 \ominus \dots \ominus B_n$ for $i=2,3$. Therefore there exist $q_2, q_3 \in R_2$ such that $\div q_i = (\div q_i)_\RR = 2(B_1 + \dots + B_n + A_i)$ for $i=2,3$. As $T_2, T_3$ are non-positive $2$-torsion points, then $q_2, q_3$ change sign on $X(\RR)$. In particular, we can assume that $q_2, q_3$ are nonnegative on $Y_1$ and nonpositive on $Y_2$. We then define $q = q_1 + q_2$. Notice that:
    \begin{enumerate}
        \item $q$ vanishes with multiplicity two at $B_1, \dots , B_n$, since both $q_2$ and $q_3$ vanish there with multiplicity two; 
        \item $q$ in nonnegative on $Y_1$, since both $q_2$ and $q_3$ are nonnegative on $Y_1$;
        \item As $A_2 \neq A_3$, $q$ does not vanish at any point of $X(\RR)$ except $B_1, \dots , B_n$, and $q$ is strictly negative on $Y_2$.
    \end{enumerate}
    This concludes the proof.
\end{proof}
\section{Carath\'eodory numbers for genus one curves}\label{sec:cara}
We refer to \Cref{sec:moment_prob} for the notations and terminology used in this section.
We investigate the Carath\'eodory number $\car = \car_{X, 2}$ of $\pos^\vee = \pos_{X, 2}^\vee$, where $X \subset \PP^n$ is a totally real elliptic normal curve. We show that the Carath\'eodory number depends on the topology of the real locus $X(\RR)$, using a technique inspired by Hilbert's proof on ternary quartics (see \cite{blekhermanLowRankSumofSquaresRepresentations2019} for a modern exposition).

Using standard conic duality, we can use \Cref{thm:elliptic_normal} to describe linear functionals $L\in \partial \pos^\vee$ on the boundary of the moment cone.
Indeed, if $L\in \partial \pos^\vee$ then by conic duality there exists $q \in \pos$ such that $L(q) = 0$.
More precisely, \Cref{thm:elliptic_normal} implies the following corollary.

\begin{corollary}\label{cor:boundary_moment_cone}
    Let $X \subset \PP^n$ be a totally real elliptic normal curve. Let $\{ \,O, T \, \} = \jac(\RR)_2^+$ be the positive $2$-torsion points. We have
    \begin{align*}
    \partial \pos^\vee = & \left\{ \, \sum_{i=1}^k \lambda_i \, \eval_{A_i} \mid  1\le k \le n, \ A_1, \dots , A_k \in X(\RR), \ \lambda_1, \dots , \lambda_k \in \RR_{>0} \, \right\} \\ & \sqcup \left\{ \, \sum_{i=1}^{n+1} \lambda_i \, \eval_{A_i} \mid \ A_1, \dots , A_{n+1} \in X(\RR) \textnormal{ distinct}, \ \lambda_1, \dots , \lambda_{n+1} \in \RR_{>0}, \ A_1 \oplus \dots \oplus A_{n+1} = O \, \right\} \\ & \sqcup \left\{ \, \sum_{i=1}^{n+1} \lambda_i \, \eval_{A_i} \mid \ A_1, \dots , A_{n+1} \in X(\RR) \textnormal{ distinct}, \ \lambda_1, \dots , \lambda_{n+1} \in \RR_{>0}, \ A_1 \oplus \dots \oplus A_{n+1} = T \, \right\}
\end{align*}
\end{corollary}

\subsection{Critical values and boundary points}
\label{sec:critical}
    Let $X \subset \PP^n$ be a totally real elliptic normal curve. For $Y \subset X(\RR)$ a connected component of $X(\RR)$, let $\pos_{Y, 2d} \subset R_{2d}$ be the cone of nonnegative forms on $Y$. As $Y$ is not (Euclidean) dense in $X(\RR)$, we have $\pos_{Y, 2d} \supsetneq \pos_{X, 2d}$ and $\pos_{Y, 2d}^\vee \subsetneq \pos_{X, 2d}^\vee$. 
    
We will use the following notation:
\begin{itemize}
    \item $Z \coloneqq v_{n,2}(X(\RR))$;
    \item If $S$ is any set, $$\cone_k(S) \coloneqq \left\{ \, \sum_{i=1}^k \lambda_i s_i \colon s_i \in S, \ \lambda_i \in \RR_{\ge 0} \, \right\}$$ denotes the set of conic combinations of points of $S$ which use at most $k$ points from $S$.
\end{itemize}

As we have $v_{n, 2}(A) \cong \eval_A \in R_2^*$, then we can characterize the  Carath\'eodory number as:
 \[\car = \car_{X,2} = \min \left\{ \, k \in \NN \mid \cone_k(Z) = \pos^\vee \, \right\}\]

\begin{lemma}\label{lem:car_ineq}
    Let $X \subset \PP^n$ be a totally real elliptic normal curve. Then $n+1 = \deg X
\le \car_{X,2} \le \deg X +1 = n+2$.
\end{lemma}
\begin{proof}
    The proof is a consequence of the following standard argument, see e.g. \cite[Th.~4.8]{didioMultidimensionalTruncatedMoment2021a}.

    Let $L = \sum_{i=1}^{n+1} \eval_{A_i}$ be such that the $A_i$ are distinct and $A_1\oplus \dots \oplus A_{n+1} \in \jac(\RR)_2^+$.
    Then $L \in \partial (\pos^\vee)$ from \Cref{cor:boundary_moment_cone}. We deduce e.g. from \cite[Prop.~18.12]{schmudgenMomentProblem2017a} that $\car(L) = n+1$, and thus $\car \ge n+1$. \Cref{cor:boundary_moment_cone} shows also that $\car(L) \le n+1$ for all $L \in \partial \pos^\vee$. Now, let $L \in \interior (\pos^\vee)$. Since $\pos^\vee$ is pointed and closed, for all $A \in X(\RR)$ there exists $\lambda_A \in \RR_{\ge 0}$ such that $L - \lambda_A \eval_A \in \partial \pos^\vee$. By the above, $\car(L - \lambda \eval_A) \le n+1$ holds for all $A \in X(\RR)$, and thus $\car(L) \le n+2$. 
\end{proof}

An equivalent result to \Cref{lem:car_ineq} was proven in \cite[Th.~4.8]{didioMultidimensionalTruncatedMoment2021a} for general affine curves with compact real locus, and considering forms of sufficiently high degree $2d$.
\Cref{lem:car_ineq} shows that there are only two possibilities for the Carath\'eodory number: In the following, we show that the exact value depends on the topology of $X(\RR)$. For the proof, we will use the following simple observation.
\begin{lemma}\label{lem:topology}
    Let $A\subset B \subset \RR^N$ be such that $A$ is closed,  $\interior A \neq \emptyset$ and $B$ is convex. Then $A=B$ if and only if $\partial A \subset \partial B$.
\end{lemma}
\begin{proof}
    Clearly, if $A = B$ then $\partial A = \partial B$.

    For the converse implication, we show that $A \neq B$ implies $\partial A \subsetneq \partial B$. Let $x \in \interior A \subset \interior B$ and $y \in B \setminus A$. Consider now the segment between $x$ and $y$. Since $B$ is convex, this segment is contained in $B$, and all the points in the relative interior of the segment belong to the interior of $B$, see e.g. \cite[Th.~6.1]{rockafellarConvexAnalysis1970a}. As $A$ is closed, $x \in \interior A$ and $y \in B \setminus A$, we can find $z$ in the relative interior of the segment such that $z \in \partial A$. Therefore, $z \in \partial A \cap \interior B$, proving that $\partial A \subsetneq \partial B$.
\end{proof}

We are now ready to determine the Carath\'eodory number in the case of connected curves, or more generally for the moment cone on connected components of the real locus.

\begin{proposition}\label{prop:car_connected}
    Let $X \subset \PP^n$ be a totally real elliptic normal curve, and let $Y$ be a connected component of $X(\RR)$. Then $\pos_{Y, 2}^\vee = \cone_{n+1}(v_{n, 2}(Y))$, or, in other words, $\car_{Y,2} = \deg X = n+1$.
\end{proposition}
\begin{proof}
    Let $Z = v_{n, 2}(Y)$. Our goal is to apply \Cref{lem:topology} to $\cone_{n+1}(Z) \subset \pos_{Y,2}^{\vee}$, to show that they are equal. First, notice that $\pos_{Y,2}^{\vee}$ is convex and that $\cone_{n+1}(Z)$ is closed. Second, $\cone_{n+1}(Z)$ has nonempty interior. Indeed, the proof of \cite[Th.~2]{blekhermanMaximumTypicalGeneric2015}, which shows that the smallest typical rank (with arbitrary real coefficients) is equal to the complex generic rank, can be generalized to the case of nonnegative coefficients (see also \cite[Lem.~37]{qiSemialgebraicGeometryNonnegative2016}). As $n+1$ is the complex generic rank (see \cite{langeHigherSecantVarieties1984} and also \Cref{sec:waring}), the above implies that $\cone_{n+1}(Z)$ has nonempty interior.
    
    Therefore, in order to apply \Cref{lem:topology} we only need to show that $\partial \cone_{n+1}(Z) \subset \partial (\pos_{Y,2})^{\vee}$. For this, we proceed similarly to Hilbert's proof that every ternary quartic is a sum of at most three squares, see e.g. \cite{blekhermanLowRankSumofSquaresRepresentations2019}. To make the analogy clearer, we identify the space $R_2^*$ with $\cI(X)^\perp \subset \RR[\vb x]_2$ using the apolar or Bombieri-Weil inner product $\BW{\,\cdot\,}{\,\cdot\,}$, see \Cref{sec:waring} for the definition and more details. If $B = (B_0, \dots , B_n), v = (v_0, \dots , v_n) \in \RR^{n+1}$ and $\ell_B = B_0 x_0 + \dots +B_n x_n$, then $\BW{\ell_B^2}{q} = \eval_{B}(q)=q(B)$ and $\BW{2\ell_v \ell_B}{q} = \partial_v q(B)$ for all $q \in \RR[\vb x]_2$.

    We can therefore consider the maps\newline
    \begin{minipage}{0.45\textwidth}
        \begin{align*}
            \phi \colon (\RR^{n+1})^{n+1} & \longrightarrow \RR[\vb x]_2 \\
            \vb A = (A_1, \dots ,  A_{n+1}) & \longmapsto \sum_{i=1}^{n+1} \ell_{A_i}^2
        \end{align*}    
    \end{minipage}
    \begin{minipage}{0.45\textwidth}
        \begin{align*}
            \psi \colon \widehat Y^{n+1} & \longrightarrow \cI(X)_2^\perp \cong R_2^* \\
            \vb A = (A_1, \dots ,  A_{n+1}) & \longmapsto \sum_{i=0}^{n} \ell_{A_i}^2
        \end{align*}  
    \end{minipage}\\
    By definition, $\psi$ is the restriction of $\phi$ to the affine cone over $Y$. Our goal is to show that $\partial \cone_3(Z) \cong \partial \Im \psi$ is included in $\partial(\pos_{Y, 2})^\vee$. The differential of $\phi$ is well-known (see e.g \cite[Lem.~2.2]{blekhermanLowRankSumofSquaresRepresentations2019}), and therefore, by restriction we have:
    \begin{align*}
        \dd \psi_{\vb A} \colon \mathrm{T}_{\vb A} (\widehat Y^{n+1}) \cong \mathrm{T}_{A_1} \widehat Y \times \dots \times \mathrm{T}_{A_{n+1}} \widehat Y  & \longrightarrow \mathrm{T}_{\psi(\vb A)}(\cI(X)_2^\perp) \cong \cI(X)_2^\perp\\
        \vb v = (v_1, \dots v_{n+1}) & \longmapsto 2\sum_{i=1}^{n+1} \ell_{v_i} \ell_{A_i}
     \end{align*}
    Notice that $\dim \mathrm{T}_{\vb A} (\widehat Y)^{n+1} = \dim \cI(X)_2^\perp = 2n+2$, and therefore $\dd \psi_{\vb A}$ is a square linear map. In particular, if $\psi(\vb A) \in \partial \Im \psi$ then $\dd \psi_{\vb A}$ does not have full rank. Using again the Bombieri-Weil inner product, this is equivalent to the fact that there exists $q \in \RR[\vb x]_2 \setminus \cI(X)_2^\perp$ such that $\BW{\ell_{v_i} \ell_{A_i}}{q} = 0$ for all $\vb v$. 
    Setting $v_i=A_i$ (this is possible since $\widehat Y$ is a cone), we have $\BW{\ell_{A_i}^2}{q} = q(A_i) = 0$ for all $i=1, \dots n+1$. Letting $v_i$ vary in the tangent space, we see that $q$ double vanishes at $A_1, \dots , A_{n+1}$. 

    There are then two cases.
    \begin{enumerate}
        \item If the $A_i$'s are distinct, $q$ (now seen as a nonzero element of $R_2 = \RR[\vb x]_2\big/ \cI(X)_2$) is uniquely determined by the double vanishing at $A_1, \dots , A_{n+1}$, i.e. $\div q = (\div q)_{\RR} = 2(A_1 + \dots + A_{n+1})$. Since $Y$ is connected, $q$ does not change sign on $Y$, and therefore we can take $q$ nonnegative on $Y$. By conic duality, this means that $\sum_{i=1}^{n+1} \eval_{A_i} \cong \sum_{i=1}^{n+1} \ell_{A_i}^2 \in \partial (\pos_{Y, 2})^\vee$. 
        \item If the $A_i$'s are not distinct, we can always find a linear form $\ell \in R_1$ vanishing on all of them, and we can take $q = \ell^2$. As in the previous point, since $q = \ell^2$ is nonnegative on $Y$ we have $\sum_{i=1}^{n+1} \eval_{A_i} \cong \sum_{i=1}^{n+1} \ell_{A_i}^2 \in \partial (\pos_{Y, 2})^\vee$.
    \end{enumerate}
    The above two cases cover all the possible $\vb A$ such that $\psi(\vb A) \in \partial \Im \psi \cong \partial \cone_{n+1}(Z)$. We have then shown that $\partial \cone_{n+1}(Z) \subset \partial (\pos_{Y, 2})^\vee$. We can therefore use \Cref{lem:topology} to conclude that $\cone_{n+1}(Z) = (\pos_{Y,2})^{\vee}$.
    In other words, we have proven that $\car_{Y, 2} = n+1$.
\end{proof}

We are now ready to prove the main result of this section.

\begin{theorem}\label{thm:car}
Let $X \subset \PP^n$ be a totally real elliptic normal curve. Then:
\begin{enumerate}
    \item if $X(\RR)$ is connected, then $\car_{X,2} = \deg X = n+1$;
    \item if $X(\RR)$ is disconnected, then $\car_{X,2} = \deg X +1 = n+2$.
\end{enumerate}
Moreover the set of $L \in \pos^\vee$ such that $\car(L) = \car_{X,2}$ has nonempty interior.
\end{theorem}
\begin{proof}
    Assume that $X(\RR)$ is connected. Then the result follows directly from \Cref{prop:car_connected}.

    Now assume that $X(\RR)$ is disconnected, and set $Z = \nu_{n, 2}(X(\RR))$. By \Cref{lem:car_ineq} we only need to show that $\car = \car_{X,2} > n+1$, or in other words that \[\cone_{n+1}(\nu_{n, 2}(X(\RR))) = \cone_{n+1}(Z) \subsetneq \pos_{X,2}^\vee = \pos^\vee.\]
    While in the proof of \Cref{prop:car_connected} we have shown that $\partial \cone_{n+1}(Z) \subset \partial (\pos_{Y,2}^{\vee})$ if $X(\RR)$ is connected, hereafter we want to show that $\partial \cone_{n+1}(Z) \subsetneq \partial (\pos_{X,2}^{\vee})$, i.e. that there exists $L \in \partial \cone_{n+1}(Z) \cap \interior (\pos_{X,2}^{\vee})$. This imples that $\cone_{n+1}(Z) \subsetneq \pos_{X,2}$, which is our claim. In order to prove the existence of $L \in \partial \cone_{n+1}(Z) \cap \interior (\pos_{X,2}^{\vee})$, we are going to use the fact that if $X(\RR)$ is disconnected, then there exists non-positive $2$-torsion points, see \Cref{lem:torsion_genus_components,lem:positive_torsion}, through \Cref{lem:positive_oval}.
    
    Let $Y_1$ be a connected component of $X(\RR)$, and let $Z_1 = v_{n, 2}(Y_1)$. By \Cref{lem:positive_oval}, there exists $q \in \pos_{Y_1, 2} \setminus \pos_{X, 2}$ such that $\div q = (\div q)_{\RR} = 2(A_1 + \dots + A_{n+1})$, with $A_i \in Y_1$ distinct.
    Let $L = \eval_{A_1} + \dots + \eval_{A_{n+1}} \in \cone_{n+1}(Z_1)$. We notice that $L \in \partial (\pos_{Y_1, 2})^\vee$, since $L(q) = 0$ and $q \in \pos_{Y_1, 2}$. Furthermore, $L \in \interior (\pos_{X,2}^\vee)$: indeed, if $L \in \partial (\pos_{X,2}^\vee)$, there would exists $\tilde q \in \pos_{X, 2}$ such that $L(\tilde q) =0$. However such a $\tilde q$ does not exist, since $A_1, \dots , A_{n+1}$ do not add to a positive $2$-torsion point. In conclusion, we have shown that
    \[
        L \in \cone_{n+1}(Z_1) \cap \partial (\pos_{Y_1, 2})^\vee \cap \interior (\pos_{X,2}^\vee)
    \]
    
    We now show that such an $L$ cannot be represented in $\cone_{n+1}(Z)$ using points in $Y_2 = X(\RR) \setminus Y_1$.

    Write $L = \sum_{i=1}^{n+1}\eval_{B_i}$, and assume that $B_1 , \dots , B_n \in Y_1$, $B_{n+1} \in Y_2 = X(\RR) \setminus Y_1$. By \Cref{lem:changing_sign}, there exists $\tilde{q} \in R_2$ such that:
    \begin{itemize}
        \item $\tilde q$ vanishes at $B_1$, \dots , $B_n$ (with multiplicity two);
        \item $\tilde q \ge 0$ on $Y_1$;
        \item $\tilde q < 0$ on $Y_2$.
    \end{itemize}
    Therefore we have:
    \begin{align*}
        L(\tilde q) & = \left(\sum_{i=1}^{n+1} \eval_{A_i}\right)(\tilde q) =  \tilde q(A_1) + \dots + \tilde q(A_{n+1}) \ge 0 \\
        L(\tilde q) & = \left(\sum_{i=1}^{n+1} \eval_{B_i}\right)(\tilde q) = \tilde q(B_1) + \dots + \tilde q(B_{n+1}) = \tilde q (B_{n+1}) < 0 
    \end{align*}
    a contradiction. A similar argument works assuming $B_1, \dots , B_k \in Y_1$ and $B_{k+1}, \dots , B_{n+1} \in Y_2$ for all $k = 0, \dots , n$.
    
    Thus any representation of $L$  as the sum of $n+1$ point evaluations only uses points of $Y_1$. This implies that, in a sufficiently small neighborhood of $L$, $\cone_{n+1}(Z)$ coincides with $\cone_{n+1}(Z_1)$. In other words, {in a sufficiently small neighborhood of } $L$, all convex combinations of $n+1$ evaluations at points of $X(\RR)$ are convex combinations of $n+1$ evaluations at points of $Y_1 \subsetneq X(\RR)$. 
    
    Furthermore, notice that $\cone_{n+1}(Z_1) = \pos_{Y_1, 2}^\vee$ by \Cref{prop:car_connected}. As $L$ defines a supporting hyperplane for $q \in \partial \pos_{Y_1, 2}$, every point in a representation of $L$ belongs to the zero locus of $q$ on $X(\RR)$. This implies that the representation $L = \eval_{A_1} + \dots + \eval_{A_{n+1}}$ is unique in $\cone_{n+1}(Z_1) = \pos_{Y_1, 2}^\vee$, and by the above it is also unique in $\cone_{n+1}(Z)$\footnote{Another way to show that the representation of $L$ is unique is to use the fact that the general point of the $(n+1)$-secant variety of an elliptic normal curve is contained in two distinct secant spaces, see \Cref{sec:waring} and \cite{chiantiniConceptKsecantOrder2006}, or equivalently that a generic $L$ has two distinct decompositions. See also \Cref{sec:waring} for more details.}.

    Therefore we see that in a sufficiently small neighborhood $U$ of $L$, we have \[\cone_{n+1}(Z) \cap U = \cone_{n+1}(Z_1) \cap U = \pos_{Y_1, 2}^\vee \cap U\] But $(\pos_{Y_1, 2})^\vee$ is convex, and thus $\cone_{n+1}(Z) \cap U$ lies completely inside one half-space defined by the tangent space $\mathrm{T}_L \partial \cone_{n+1}(Z) \cap U$ to $L \in \partial (\pos_{Y_1, 2})^\vee \cap U= \partial \cone_{n+1}(Z) \cap U$. Recall also that $L \in \interior(\pos_{X, 2}^\vee)$:
    therefore, going in the normal direction $\mathcal{N}_L \partial \cone_{n+1}(Z)$ to $\mathrm{T}_L \partial \cone_{n+1}(Z)$ we can find an open set included in $\pos_{X,2}^\vee \setminus \cone_{n+1}(Z)$, concluding the proof.
\end{proof}

As $\pos_{X, 2d} \cong \pos_{\nu_{n,d}(X), 2}$, we can extend \Cref{thm:elliptic_normal} to higher degree forms.

\begin{corollary}\label{cor:cara}
    Let $X \subset \PP^n$ be a totally real elliptic normal curve. Then:
\begin{enumerate}
    \item if $X(\RR)$ is connected, then $\car_{X,2d} = \deg X = d(n+1)$;
    \item if $X(\RR)$ is disconnected, then $\car_{X,2d} = \deg X + 1 = d(n+1)+1$.
\end{enumerate}
\end{corollary}
\begin{proof}
    Apply \Cref{thm:car} to the elliptic normal curve $\nu_{n,d}(X)$.
\end{proof}

\subsection{Discussion: Carath\'eodory numbers and Waring decompositions}
\label{sec:waring}
We hereafter summarize the equivalence between the computation of the Carath\'eodory numbers and the minimal rank of a certain Waring-type decomposition.

We briefly recall the \emph{apolar} or \emph{differential} or \emph{Bombieri-Weil} inner product \cite{iarrobinoPowerSumsGorenstein1999}.
Let $\KK = \CC$ or $\RR$. The inner product $\BW{\,\cdot\,}{\,\cdot\,}$ on $\KK[\vb x]_k = \KK[x_0, \dots , x_n]_k$ is defined on monomials $\vb x^\alpha, \vb x^\beta$ (and extended by linearity) as:
\[
    \BW{\vb x^\alpha}{\vb x^\beta} \coloneqq \frac{1}{k!} {\partial_\alpha \, \vb x^\beta} (0)
\]
One can show that, if $A = (A_0, \dots , A_k) \in \KK^{n+1}$ and $f \in \KK[\vb x]_k$, then $\BW{\ell_A^k}{f} = f(A) = \eval_A(f)$, where $\eval_A \colon \KK[\vb x]_k \to \RR$ denotes the usual point evaluation and $\ell_A = A_0x_0 + \dots + A_n x_n$. Now let $Z \subset \PP^n$ be a smooth, non-degenerate algebraic variety. If $L \in (\KK[\vb x]/I(Z))_k^* \cong I(Z)_k^\perp \subset \KK[\vb x]_k$, $A_i \in Z(\KK)$ and $a_i\in \KK$, then:
\[
    L = \sum_{i=1}^r a_i \, \eval_{A_i} \iff f_L = \sum_{i=1}^r a_i \, \ell_{A_i}^k
\]
where $f_L \in I(Z)^\perp \subset \KK[Z]$ is the unique polynomial representing $L = \BW{f_L}{\,\cdot\,}$ using the Riesz representation theorem. If $Z = \PP^n$, such a decomposition of a form as a sum of powers of linear forms is called a \emph{Waring decomposition} of rank $r$ over $\KK$ (assuming that the $A_i$'s are pairwise distinct), and the minimal such $r$ is called the \emph{Waring rank} over $\KK$ of $f$. More generally, when $Z \subset \PP^n$ is a non-degenerate algebraic variety, such a minimal $r$ is called the 
$Z$-rank of $f$ over $\KK$. If we write $Z = v_{n,2d}(X)$ for the second Veronese embedding, the Carath\'eodory number $\car_{X,2d}(L)$ corresponds to the \emph{nonnegative} $Z$-rank of $f_L$, i.e. to the $Z$-rank over $\RR_{\ge 0}$, see e.g. \cite{qiSemialgebraicGeometryNonnegative2016,blekhermanRealRankRespect2016, angeliniRealIdentifiabilityVs2018}. We say that a nonnegative $Z$-rank $r$ is \emph{typical} if the set of forms of nonnegative $Z$-rank equal to $r$ has nonempty interior. Using this terminology, \Cref{cor:cara} can be rephrased as follows.

\begin{corollary}\label{cor:waring}
    Let $X \subset \PP^n$ be a totally real elliptic normal curve and denote $Z = \nu_{n,2d}(X)$. Then, for linear functionals in the moment cone $\pos_{X,2d}^\vee$:
\begin{enumerate}
    \item if $X(\RR)$ is connected, then the maximal nonnegative $Z$-rank is $d(n+1)$;
    \item if $X(\RR)$ is disconnected, then the maximal nonnegative $Z$-rank is $d(n+1)+1$.
\end{enumerate}
Furthermore, the maximal nonnegative $Z$-rank is always typical.
\end{corollary}

Let us also remark on another connection between our study of Carath\'eodory numbers and secant varieties. The proof of \Cref{thm:car} relies on the study of the boundary points of the set of linear functionals obtained using at most $n+1$ point evaluations. In particular, we study the tangent space at the boundary, which is of a smaller dimension than the expected one. In the language of tensor decomposition, those points define what is called \emph{Terraccini locus}. For more information and precise definitions on the Terraccini locus, we refer the reader to the recent work \cite{angeliniWaringDecompositionsSpecial2023} and references therein.

We finally state some lemmas on the geometry of the representations of linear functionals as a sum of point evaluations. They will be used in \Cref{sec:dual_plane_cubics} to exclude the points at infinity from the allowed point evaluations. Their proofs are mostly applications of standard techniques from projective and convex geometry, and are provided in \Cref{app:B}. The first lemma is a particular case of \cite[Prop.~5.2]{chiantiniConceptKsecantOrder2006}.
\begin{lemma}\label{lem:chiantini}
    Let $X\subset \PP^n$ be an elliptic normal curve, $Z = v_{n,2d}(X)$, and \[L = \eval_{A_1} + \dots + \eval_{A_{d(n+1)}} \ \in \left( \CC[\vb x]_{2d}\big/I(X)_{2d}\right)^*\] If $A_1, \dots , A_{d(n+1)}$ are distinct and $A_1 \oplus \dots \oplus A_{d(n+1)}$ is not a $2$-torsion point of $X$, then $L$
    has two distinct representations of $Z$-rank (over $\CC$) equal to $d(n+1)$.
\end{lemma}

We now adapt \Cref{lem:chiantini} to the \emph{nonnegative} $Z$-rank: in \Cref{lem:chiantini_real} we treat the generic case, while in \Cref{lem:chiantini_real_special} we consider special cases.
\begin{lemma}\label{lem:chiantini_real}
    Let  $X\subset \PP^n$ be a totally real elliptic normal curve and $Z = v_{n,2d}(X(\RR))$.
    All linear functions $L \in \cone_{d(n+1)}(Z) \setminus \partial\cone_{d(n+1)} (Z)$ (which is a Euclidean dense subset of $\cone_{d(n+1)}(Z)$) admit exactly two different representations:
    \[
        L = \sum_{i=1}^{d(n+1)} a_i \, \eval_{A_i} = \sum_{i=1}^{d(n+1)} b_i \, \eval_{B_i}
    \]
    with $a_i, b_i \in \RR_{\ge 0}$ and $A_i, B_i \in X(\RR)$. 
\end{lemma}
\begin{lemma}\label{lem:chiantini_real_special}
    Let  $X\subset \PP^n$ be a totally real elliptic normal curve and $Z = v_{n,2d}(X(\RR))$.
    All linear functionals $L
    \in \partial\cone_{d(n+1)}(Z) \setminus \cone_{d(n+1)-1}(Z)$ admit at most two representations as in \Cref{lem:chiantini_real}.
\end{lemma}
In the previous lemmas, we focused on representations using $d(n+1)$ points, showing that they are generically finitely many. On the other hand, when we use $d(n+1)+1$ points a generic linear functional $L$ admits infinitely many representations.
More precisely, we can show that every linear functional in the relative interior of $\pos_{X, 2d}^\vee$ has a $2$-dimensional family of representations using $d(n+1)+1$ points.

\begin{lemma}\label{lem:fiber}
    Let  $X\subset \PP^n$ be a totally real elliptic normal curve and let $L \in \interior \pos_{X, 2d}^\vee$. Then $L$ has a $2$-dimensional family of representations using $d(n+1)+1$ point evaluations.
\end{lemma}

We now use \Cref{lem:fiber} to show that there exist representations of linear functionals in $\pos_{X,2d}^\vee$ whose atoms avoid finitely many points in $X(\RR)$.
\begin{lemma}\label{lem:excluding_points}
    Let $X\subset \PP^n$ be a totally real elliptic normal curve and
    let $L$ be in the interior of $\pos_{X, 2d}^\vee$. If \[\mathcal B = \{ \, B_1, \dots , B_{m} \} \subset \widehat X(\RR) \cap \SS^{n}\]
    then there exists a representation of $L$ using $d(n+1)+1$ atoms that do not belong to $\mathcal B$.
\end{lemma}
\begin{remark}
    We make some remarks on the previous results.
    \begin{enumerate}
        \item In the proof of \Cref{lem:chiantini_real}, we are implicitly using the fact that, if $L \in \pos_{X, 2d}^\vee$, then the \emph{moment or catalecticant matrix} $M_d(L)$ is positive definite, see \Cref{sec:almost_flat}.
        \item Assuming that $X(\RR)$ is connected, we can deduce from \Cref{cor:waring} that the rank of the moment or catalecticant matrix $M_d(L)$ is equal to the nonnegative $Z$-rank of $L$: in general, the rank of $M_d(L)$ is only a lower bound for the rank. 
        \item Given a general $L$, in \Cref{lem:chiantini_real} (and in \Cref{lem:chiantini}) the atoms of the two different possible representations form the zero locus of a form $q = q_L \in R_{2d}$ on the elliptic normal curve, i.e. \[\div q = A_1 + \dots + A_{d(n+1)}+B_1 + \dots +B_{d(n+1)}.\]
        This shows that the set of possible atoms forms a self-associated set for the Gale transform, see \cite{eisenbudProjectiveGeometryGale2000}. Therefore, given $L$ and a representing atoms $A_1, \dots , A_{d(n+1)}$, the problem of determining the second set of representing atoms $B_1, \dots , B_{d(n+1)}$ could be investigated using the geometry of the Gale transform.
    \end{enumerate}
\end{remark}

\section{The moment problem for plane cubics} \label{sec:dual_plane_cubics}

In this section, we consider \emph{affine} genus one curves, and see how the results in \Cref{sec:cara} can be exploited to solve the \emph{moment problem} on such curves. In particular, we show how the topology of the real locus and the number of points at infinity can be used to characterize the \emph{flat extension} degree. In the following, affine curves and their real loci will be denoted by $X$ and $X(\RR)$ respectively, while their projective closures will be denoted by $\overline X$ and $\overline X(\RR)$. 

In recent years, the case of \emph{plane} curves has been particularly investigated, see \cite{fialkowSolutionTruncatedMoment2011, zalarTruncatedHamburgerMoment2021, zalarTruncatedMomentProblem2022, zalarTruncatedMomentProblem2023, bhardwajNonnegativePolynomialsSums2020}. For ease of presentation and comparison, we restrict ourselves to the same case. i.e. we consider a smooth totally real affine plane cubic $X(\RR) \subset \RR^2$.

\subsection{Flat extensions and almost flat extensions}\label{sec:almost_flat}
We now briefly introduce the relevant definitions and notations for this section.

Let $X(\RR) \subset \RR^2$ be a totally real affine plane cubic.
Given \[L \in R_{\le 2d}^* = \RR[X]_{\le 2d}^* = \left(\RR[x, y]_{\le d} \big/I(X)_{\le d}\right)^*\] we can consider the bilinear form
\begin{align*}
    M_d(L) \colon R_{\le d} \times R_{\le d} & \longrightarrow \RR \\
    (q_1, q_2) & \longmapsto L(q_1q_2)
\end{align*}
Equivalently, $M_d(L)$ is the linear map:
\begin{align*}
    M_d(L) \colon R_{\le d} & \longrightarrow R_{\le d}^* \\
    q & \longmapsto L \circ {m}_q
\end{align*}
where ${m}_q \colon R_{\le d} \to R_{\le 2d}$ denotes the multiplication by $q\in R_{\le d}$.
We call the matrix representing $M_d(L)$ (with respect to any basis of $R_{\le d}$) the \emph{moment matrix} of $L$. We refer to \cite{schmudgenMomentProblem2017a} for more details. These matrices are also known as \emph{Catalecticant matrices} \cite{iarrobinoPowerSumsGorenstein1999} (in the case $X = \RR^n$) or \emph{Hankel matrices} \cite{mourrainPolynomialExponentialDecomposition2018a}. 
We also notice that $\rank M_d(L) \le \dim R_{\le d} = 3d$.

Similarly to the homogeneous case, we denote \[\pos_{X, \le 2d} = \{ \, q \in R_{\le 2d} \colon q(A) \ge 0 \text{ for all } A\in X(\RR) \,\}.\] 
Notice that, if $L \in \pos_{X, \le 2d}^\vee$, then the moment matrix $M_d(L)$ is positive semidefinite.

As $R_{2d} \cong R_{\le 2d}$ by homogenization, we also have $\pos_{X, \le 2d} \cong \pos_{\overline{X}, 2d}$ and $\pos_{X, \le 2d}^\vee \cong \pos_{\overline{X}, 2d}^\vee$. To emphasize the fact that we are working in an affine setting, hereafter we use the notations $R_{\le 2d}$ and $\pos_{X, \le 2d}$.

We denote $\mcone_{2d}(X(\RR))$ the \emph{moment cone}, i.e. the convex cone of linear functionals acting on polynomials of degree $\le 2d$ which are induced by a Borel measure supported on $X(\RR)$. Using the Richter-Tchakaloff theorem \cite[Th.~1.24] {schmudgenMomentProblem2017a}, we can describe $\mcone_{2d}(X(\RR))$ using only conic sums of point evaluations at $X(\RR)$:
\[
    \mcone_{2d}(X(\RR)) = \cone(\eval_A \colon A \in X(\RR))
\]
If $X(\RR)$ is compact then $\mcone_{2d}(X(\RR)) = \pos_{X, 2d}^\vee$. On the other hand, if $X(\RR)$ is not compact, then $\mcone_{2d}(X(\RR))$ is not closed and the inclusion $\mcone_{2d}(X(\RR)) \subsetneq \pos_{X, \le 2d}^\vee$ is proper: this is always the case for plane cubics $X(\RR) \subset \RR^2$.

We are now ready to recall the notion of flat extension.

\begin{definition}\label{def:flat_extension}
    A \emph{positive flat extension} of $L \in R_{\le 2d}^*$ is any $\widetilde L \in R_{\le 2d+2}^*$ such that:
\begin{enumerate}
    \item $\widetilde L\mid_{{R_{\le 2d}}} = L$;
    \item $\rank M_{d+1}(\widetilde L) = \rank M_d(L)$;
    \item $M_{d+1}(\widetilde L)$ is positive semidefinite.
\end{enumerate}
\end{definition}
Testing the existence of a flat extension is the standard algorithmic tool to verify if $L$ is a moment linear functional or not, see e.g. \cite{curtoFlatExtensionsPositive1998,schmudgenMomentProblem2017a, mourrainPolynomialExponentialDecomposition2018a}.
Our definition of flat extension is slightly different from others present in the literature for the case of curves (see e.g. \cite{fialkowSolutionTruncatedMoment2011}), in terms of how we enforce the atoms of any representation of $L$ to lie on $X(\RR)$. We refer to \Cref{rem:moment_dictionary} for a comparison.

The rank of the flat extension determines the number of atoms needed to represent $L$ as conic combination of point evaluations on $X(\RR)$, see e.g. \cite[Th.~17.36]{schmudgenMomentProblem2017a}.
Moreover, if a flat extension is found, then the flat extension can be further extended from degree $d+1$ to an arbitrary degree. Notice also that, since $\widetilde L\mid_{{R_{\le 2d}}} = L$, the matrix $M_d(L)$ can be identified with the matrix $M_d(\widetilde L)$, which is a submatrix of $M_{d+1}(\widetilde L)$. 

In other words, the rank of a flat extension gives the (affine) Carath\'eodory number of a linear functional. We define the Carath\'eodory number for the affine plane curve $X$ as in the projective case, as follows:
\[
    \car_{X, \le 2d} \coloneqq \min \left\{ \, k \in \NN \mid \cone_k(\nu_{2, d}(X(\RR)))= \mcone_{2d}(X(\RR)) \, \right\}
\]
Note that we are not allowed to use atoms at infinity, i.e. point evaluations from $\overline X(\RR) \setminus X(\RR)$.

In the following, we will also need a slight generalization of \Cref{def:flat_extension}.

\begin{definition}\label{def:almost_flat}
     An \emph{almost flat extension} of $L \in R_{\le 2d}$ is any $\widetilde L \in R_{\le 2d+4}^*$ such that:
\begin{enumerate}
    \item $\widetilde L\mid_{{R_{\le 2d}}} = L$;
    \item $\rank M_{d+2}(\widetilde L) \le \rank M_d(L) + 1$;
    \item $M_{d+2}(\widetilde L)$ is positive semidefinite.
\end{enumerate}
\end{definition}

\begin{remark}
    If $\widetilde L$ is a flat extension of $L$, then $\widetilde L$ can be further extended to an almost flat extension of $L$. More precisely, if $L$ admits a flat extension $\widetilde L$, then $\widetilde L \in R_{\le 2d+2}^*$ admits a flat extension $\widetilde{\widetilde L} \in R_{\le 2d+4}^*$ such that:
\[
    \rank M_{d+2}(\widetilde{\widetilde L}) = \rank M_{d+1}(\widetilde L) = \rank M_d(L).
\]
On the contrary, if $L$ does not admit a flat extension, but it admits an almost flat extension $\widetilde L$, we necessarily have:
\[
    \rank M_{d+2}(\widetilde L) = \rank M_{d+1}(\widetilde L) = \rank M_d(L) + 1.
\]
\end{remark}

In \Cref{sec:moment_problem_for_cubics} we show that we can characterize solutions to the moment problem on plane cubics using almost flat extensions.
\begin{remark}\label{rem:moment_dictionary}
    We compare our definitions with the ones existing in the moment literature, e.g. in \cite[Th.~1.1]{fialkowSolutionTruncatedMoment2011}. The main difference is that, instead of working with all the monomials up to a certain degree and imposing extra conditions for the representing measure to be supported on the plane curve, we directly work in the quotient $R_{\le 2d} = \RR[x, y]_{\le 2d} \big/I(X)_{\le 2d}$. We also speak of the flat extension of a linear functional, instead of a flat extension of the corresponding moment matrix.
    Concretely, a \emph{p-pure} moment matrix in \cite[Th.~1.1]{fialkowSolutionTruncatedMoment2011} corresponds to a positive definite moment matrix in this manuscript, and a \emph{positive, recursively generated flat moment matrix extension} in \cite[Th.~1.1]{fialkowSolutionTruncatedMoment2011} simply corresponds to our flat extension.
\end{remark}

\subsection{The moment problem for plane cubics} \label{sec:moment_problem_for_cubics}
We are now ready to show how to solve the moment problem, i.e. how to characterize membership $L \in \mcone_{2d}(X(\RR))$, using almost flat extensions, see \Cref{def:almost_flat}.

\begin{proposition}\label{prop:almost_flat}
    Let $X(\RR) \subset \RR^2$ be the affine real locus of a totally real plane cubic, whose projectivization $\overline{X}$ is smooth. Given $L \in R_{\le 2d}^*$, the following are equivalent:
    \begin{enumerate}
        \item $L \in \mcone_{2d}(X(\RR))$;
        \item $l$ has an almost flat extension $\widetilde L \in R_{\le 2d+4}$.
    \end{enumerate}
\end{proposition}
\begin{proof}
    $(i)\implies (ii)$ Recall that, by homogenization, $\pos_{X, \le 2d} \cong \pos_{\overline X , 2d}$. Therefore, given $L \in \mcone_{2d}(X(\RR)) \subset \pos_{X, \le 2d}^\vee$, it follows from \Cref{cor:cara} that we can always write \[ L = \sum_{i=1}^{3d+1} a_i \, \eval_{A_i}\]
    where, for all $i$, $a_i \ge 0$ and $A_i$ belongs to $\overline{X}(\RR) \subset \PP^2(\RR)$, the real locus of the projective closure $\overline X$ of $X$. From \Cref{lem:excluding_points}, we deduce that there exists a representation for which $A_i \in X(\RR)$ (we are excluding the finitely many points at infinity).
    Therefore $L$ admits the almost flat extension $\widetilde L = \sum_{i=1}^{3d+1} a_i \, \eval_{A_i} \in R_{2d+4}^*$.

    $(i)\implies (ii)$ If $L$ admits an almost flat extension $\widetilde L$, then either $\rank M_{d+1}(\widetilde L) = \rank M_d(L)$ or $\rank M_{d+2}(\widetilde L) = \rank M_{d+1}(\widetilde L)$. From the flat extension theorem, we deduce that, in both cases, $L$ can be represented as an atomic measure using point evaluations in $X(\RR)$, which implies $L \in \mcone_{2d}(X(\RR))$. 
\end{proof}

We showed in \Cref{prop:almost_flat} that, extending the degree by $4$ and considering an almost flat extension, it is possible to characterize whether $L \in \mcone_{2d}(X(\RR))$ or not. A natural question then arises: is it necessary to consider almost flat extensions, or can we replace them with ordinary flat extensions? In the following, we are going to show how the answer to this question depends on the topology of $X(\RR)$ and the number of points at infinity of the projective closure $\overline X (\RR)$.

\begin{theorem}\label{thm:flat_conn_one_infty}
    Let $X$ be an affine, smooth, totally real plane cubic. Assume that $X(\RR)$ is connected and that the projective closure $\overline X (\RR)$ contains only one real point at infinity. Then, given $L \in R_{\le 2d}$, the following are equivalent:
    \begin{enumerate}
        \item $L \in \mcone_{2d}(X(\RR))$;
        \item $L$ has a positive flat extension $\widetilde L \in R_{\le 2d+2}^*$.
    \end{enumerate}
    Moreover, the affine and projective Carath\'eodory numbers are equal: $\car_{X, \le 2d}= \car_{\overline X, 2d}=3d$.
\end{theorem}
\begin{proof}
    For $(ii) \implies (i)$, we can proceed as in \Cref{prop:almost_flat}.

    For $(i) \implies (ii)$, it follows from \Cref{thm:car} that $L$ can be written using at most $3d$ points from $\overline X(\RR)$. We now show that at least one such representation has all points in the affine part $X(\RR)$.

    If $L \in \partial \mcone_{2d}(X(\RR))$, it follows e.g. from \cite[Prop.~18.12]{schmudgenMomentProblem2017a} that $L$ has a unique representation as a sum of point evaluations in $X(\RR)$. This unique representation uses at most $3d$ point evaluations, and thus we can construct a flat extension of $L$ using such a representation.

    If $L \in \interior \mcone_{2d}(X(\RR))$, then \Cref{lem:chiantini_real} implies that $L$ admits two different representations as sum of $3d$ point evaluations. Such distinct representations cannot share any point, and therefore the fact that $\overline X (\RR)$ has a unique point at infinity implies that one of the two uses only points in $X(\RR)$. We can therefore construct a flat extension for $L$ from this representation with only points in $X(\RR)$.

    The fact that $\car_{X, \le 2d}= \car_{\overline X, 2d}$ follows from the previous points, as the existence of representations using points from $\overline{X}(\RR)$ implies that existence of representations using points from $X(\RR)$ with the same number of atoms.
\end{proof}
We now construct an example which illustrates that $\overline X (\RR)$ having a unique point at infinity is a necessary assumption in \Cref{thm:flat_conn_one_infty}.
\begin{example}\label{ex:almost_flat}
    Let $\widetilde X(\RR) \subset \RR^2$ be the real locus of the affine totally real plane cubic, and assume that $\widetilde X(\RR)$ is connected. Let $L \in \interior \mcone_{2d}(\widetilde X(\RR)) \subset \pos_{\widetilde X, \le 2d}^\vee$.

    We deduce from \Cref{lem:chiantini_real} that $L$ admits exactly two distinct representations
    \[
        L = \sum_{i=1}^{3d} a_i \, \eval_{\widetilde A_i} = \sum_{i=1}^{3d} b_i \, \eval_{\widetilde B_i}
    \]
    using $3d$ points form the projective closure of $\widetilde X(\RR)$.
    Now consider the (real) line between $A_1$ and $B_1$, and apply a projective linear change of coordinates in such a way the line between $\widetilde A_1$ and $\widetilde B_1$ is sent to the line at infinity.
    Let $X(\RR) \subset \RR^2$ be the affine real locus of $\widetilde X(\RR)$ under this change of coordinates.

    We have $L \in \interior \mcone_{2d}(X(\RR)) = \interior \mcone_{2d}(\widetilde X(\RR))$, and as before $L$ admits only two representations
        \[
        L = \sum_{i=1}^{3d} a_i \, \eval_{A_i} = \sum_{i=1}^{3d} b_i \, \eval_{B_i} \in \pos_{\widetilde X, \le 2d}^\vee = \pos_{X, \le 2d}^\vee
    \]
    where $A_i, B_i$ denote the images of $\widetilde A_i, \widetilde B_i$ after the change of coordinates above. Notice that both representations use a point at infinity, i.e. $A_1$ and $B_1$ respectively, by construction.

    This shows that $L \in \interior \mcone_{2d}(X(\RR))$ cannot be represented using $3d$ evaluations at points of the affine curve $X(\RR)$, and thus $L$ does not admit a flat extension. Therefore, we cannot remove the condition that $\overline X(\RR)$ has a unique point at infinity from \Cref{thm:flat_conn_one_infty}, to conclude that the flat extension condition is sufficient.
    Notice also that the above discussion implies that $\car_{X, \le 2d}(L) = 3d+1$, but $\car_{\overline X, 2d} = 3d$, and thus $\car_{\overline X , 2d} < \car_{X, \le 2d}$. It is therefore not possible to compute the affine Carath\'eodory numbers using solely the projective information.
    
    We conjecture that such pathological behavior happens for all affine plane cubics with at least two real points at infinity.
\end{example}

We now turn our attention to the disconnected case.
\begin{theorem}\label{thm:almost_flat}
    Let $X(\RR) \subset \RR^2$ be the affine real locus of a totally real plane cubic, whose projectivization $\overline{X}$ is smooth. Assume that $X(\RR)$ has two connected components. Then, given $L \in R_{\le 2d}$, the following are equivalent:
        \begin{enumerate}
        \item $L \in \mcone_{2d}(X(\RR))$;
        \item $L$ has a positive almost flat extension $\widetilde L \in R_{\le 2d+4}^*$.
    \end{enumerate}
    The almost flat extension condition cannot be replaced by flat extension in (ii). Moreover, the affine and projective Carath\'eodory numbers are equal: $\car_{X, \le 2d}= \car_{\overline X, 2d}=3d+1$.
\end{theorem}
\begin{proof}
    By \Cref{prop:almost_flat}, we only need to show that there exists $L \in \mcone_{2d}(X(\RR))$ which does not admit a flat extension. For this, it is sufficient to show that $\car_{\overline X , 2d}(L) > 3d$. By \Cref{thm:car}, such an $L$ exists when $X(\RR)$ has two connected components. This also implies that $\car_{X, \le 2d}= \car_{\overline X, 2d}=3d+1$.
\end{proof}
We notice that the solution of the moment problem proposed in \Cref{thm:almost_flat} with the almost flat extension condition is much more complicated than the usual flat extension one, exploited in \Cref{thm:flat_conn_one_infty} and \cite{fialkowSolutionTruncatedMoment2011, zalarTruncatedHamburgerMoment2021, zalarTruncatedMomentProblem2022, zalarTruncatedMomentProblem2023, bhardwajNonnegativePolynomialsSums2020}. However, this increased complexity cannot be avoided, as it is a direct consequence of the higher Carath\'eodory number.

\section*{Acknowledgments}
The authors wish to thank Mario Kummer, Luca Chiantini, and Aljaz Zalar for insightful discussions on \Cref{sec:elliptic_normal}, \Cref{sec:waring}, and \Cref{sec:dual_plane_cubics}, respectively, and Mauricio Velasco for his useful comments.

Part of this project was developed by the authors at the Oberwolfach Research Institute for Mathematics, during Workshops number 2311 and 2312.
Lorenzo Baldi was partially funded by the Paris \^{I}le-de-France Region, under the grant agreement 2021-02--C21/1131, and by the Humboldt Research Fellowship for postdoctoral researchers. Grigoriy Blekherman was partially supported by NSF grant DMS-1901950.

\appendix

\section{Plane cubic curves}
We begin with the example of nonnegative quadrics on cubic plane curves, where we illustrate our results in detail {and provide explicit certificates on nonnegativity obtained by a direct computation}.
\subsection{Plane cubics}\label{sec:plane_cubics}
Plane cubics are the first example of genus one curves. For their basic properties, we refer e.g. to \cite{shafarevichBasicAlgebraicGeometry2013, fultonAlgebraicCurves1989}. In this case, the analysis of $\pos_{X, 2}$ can be performed explicitly and completely. 

Plane cubics are curves in $\PP^2$ defined as the zero locus of a form of degree three. In the following, we always assume that our plane cubics are smooth, irreducible, and totally real.
Recall that the group law of a plane cubic $X$ is induced by the isomorphism
between the plane cubic and the Jacobian:
\begin{align*}
\varphi_{P_{0}}: X & \longrightarrow \Cl^{0} X \cong \jac \\
P & \longrightarrow [P-P_{0}]
\end{align*}
Hereafter, we consider $P_{0}=O$ to be a smooth real inflection point,
so that our plane cubic has equation (Weierstrass form):
\[
h = x_2^{2}x_0-\left(x_1-a_{1}x_0\right)\left(x_1-a_{2}x_0\right)\left(x_1-a_{3}x_0\right) = 0
\]
with $a_{i} \in \CC$. $X$ is non-singular if and only if
$a_{i}\in \CC$'s are distinct.

Denote $\oplus$ and $\ominus$ the group operations in $X$. Then the point at infinity $O$ is the identity,
and \[T_{1}=\left(1:a_{1}: 0\right), T_{2}=\left(1:a_{2}: 0\right), T_{3}=\left(1:a_{3}:0\right)\] are
the $3$ nontrivial $2$-torsion points for $(X, \oplus)$ (see also \Cref{def:two_torsion}).

Since we are assuming that $X$ is totally real, then at least one of the $a_{i}$ is real,
say $a_{1}$, while $a_{2}$ and $a_{3}$ can be either:
\begin{enumerate}
  \item Complex conjugates $a_{2}=\overline{a_{3}}$,
  and in this case $X(\RR)$ is connected in the Euclidean topology; or
  \item both real (and distinct, if $X$ is smooth),
  and in this case $X(\RR)$
  has two connected components. In such a case, we will always assume without loss of generality that $a_1<a_2<a_3$.
\end{enumerate}
These are the two possible topologies for a genus one totally real curve. In the first case, there are only two real $2$-torsion points, $O$ and $T_1$, while in the second case there are four. This is consistent with \Cref{lem:torsion_genus_components}.

We focus on quadratic forms on $X$, and we start by analyzing the extreme rays of $\pos_{X, 2}$, which are given by forms with the maximal possible number of real zeroes. Given a quadric $q \in \RR[x_0, x_1, x_2]_2 \cong \RR[x_0, x_1, x_2]_2/(h)_2 = \RR[X]_2$, we denote $\div q = h\cdot q$ the intersection divisor of $q$ and $h$ (or of $q$ and $X = \cV(h)$), see e.g. \cite{fultonAlgebraicCurves1989}. If $q$ is nonnegative on $X(\RR)$, then $q$ has zeroes on $X(\RR)$ with even multiplicity only. By B\'ezout's theorem, there are at most three such zeroes. We then deduce from \cite[Problem~5.41]{fultonAlgebraicCurves1989} that, given $A_1, A_2, A_3 \in X(\RR)$, there exists $q\in \RR[\vb x]_2$ such that $\div q = 2(A_1+A_2+A_3)$ if and only $A_1\oplus A_2 \oplus A_3$ is a 2-torsion point for $(X, \oplus)$.

If $X(\RR)$ is connected, then $\div q = 2(A_1+A_2+A_3)$ implies that $q$ is nonnegative.
On the other hand, if $X(\RR)$ is disconnected then $\div q = 2(A_1+A_2+A_3)$ does not imply that $q$ is nonnegative. We now show that in the disconnected case, whether $q$ is nonnegative or not depends on the $2$-torsion point $A_1\oplus A_2 \oplus A_3$.

If $A_1 \oplus A_2 \oplus A_3 = O$, then there exists a linear form $\ell \in \RR[\vb x]_1$ such that $\div \ell = A_1+A_2+A_3$, and thus $q = a \,\ell^2$ for some $a \in \RR_{>0}$. This shows that $q$ is a sum of squares and thus nonnegative. 

We now deal with the other cases. In the following, we denote $\ell_{A_1, A_2}$ the line passing through $A_1, A_1 \in X(\RR)$, and $\ell_A$ the line tangent to $A\in X(\RR)$. Recall from the definition of the group law that the third intersection point of $\ell_{A_1, A_2}$ with $X(\RR)$ is equal to $\ominus A_1 \ominus A_2$, i.e. $\div \ell_{A_1, A_2} = A_1+A_2+(\ominus A_1 \ominus A_2)$.
Assume that $A_1\oplus A_2 \oplus A_3 = T_i$ for some $i\in \{ \, 1,2,3\, \}$ (recall that we have chosen indices in such a way $a_1<a_2<a_3$). A direct computation shows that:

\begin{equation}\label{eq:divq}
    \div q = \div\frac{\ell_{A_1, A_2}^2 \ell_{A_1 \oplus A_2, A_3}^2}{\ell_{\ominus A_1 \ominus A_2, A_1 \oplus A_2}^2 \ell_{T_i}^2}\ell_O \ell_{T_i}
\end{equation}

and since $X$ is irreducible, the two expressions agree (up to a positive constant) on $X(\RR)$. Therefore, $q$ is nonnegative if and only if $\ell_O \ell_{T_i}$ is nonnegative. For $i=2,3$ it is easy to show that $\ell_O \ell_{T_i}$ changes sign on $X(\RR)$, and thus $q$ is not nonnegative. On the other hand, for $i=1$ we have:
\begin{equation}\label{eq:divl0l1}
    (\ell_{T_1}^2+(a_2-a_1)(a_3-a_1)\ell_O^2) \, \ell_O \ell_{T_1} = x_2^2 \ell_O^2 + (a_2 + a_3 - 2 a_1)\ell_O^2 \ell_{T_1}^2
\end{equation}
which shows that $\ell_O \ell_{T_1}$, and thus $q$, is nonnegative on $X(\RR)$.

In conclusion, $A_1, A_2, A_3$ define a quadric $q$ such that $\div q = 2(A_1+A_2+A_3)$ which is nonnegative if and only if:
\begin{enumerate}
    \item either $A_1 \oplus A_2 \oplus A_3 = O$, and in this case $q = \ell^2$ is a square;
    \item or $A_1 \oplus A_2 \oplus A_3 = T$, where $T$ is the unique nontrivial real $2$-torsion point such that $\ell_O \ell_T$ is nonnegative on $X(\RR)$.
\end{enumerate}
These quadrics span extreme rays of $\pos = \pos_{X, 2}$ since they have the maximal number of real zeroes: there are therefore two families of extreme rays defined by the two positive $2$-torsion points $O$ and $T$. This is consistent with \Cref{lem:positive_torsion_genus}.

We now study other faces. In particular, recall from \Cref{thm:char_faces} that all the faces $\face \subset \pos$ are of the form $\face = \face_{A_1+\dots+A_k}$, with $A_1, \dots , A_k \in X(\RR)$.

If $k \ge 4$, then $\face_{A_1+\dots + A_k} = \{ \, 0 \, \}$ by B\'ezout's theorem. For $k=3$, we deduce from the above discussion that either $\dim \face_{A_1+A_2+A_3}=1$, when the $A_i$'s are in special position, otherwise $\face_{A_1+A_2+A_3} = \{ \, 0 \, \}$. 

For $k=2$, we can consider the two additional points $B_1 = O \ominus A_1 \ominus A_2$ and $B_2 = T \ominus A_1 \ominus A_2$. Then $\face_{A_1+A_2+B_1}$ and $\face_{A_1+A_2+B_2}$ are extreme rays, and $\face_{A_1+A_2} = \cone(\face_{A_1+A_2+B_1}, \face_{A_1+A_2+B_2})$ has dimension $2$. Similarly, for $k = 1$ we can show that 
$\dim \face_{A_1} = 4$.
In conclusion, we proved the following.

\begin{proposition} \label{prop:psd_plane_cubic}
    Let $X \subset \PP^2$ be a smooth totally real plane cubic, and denote $\pos = \pos_{X, 2}$ the convex cone of nonnegative quadratic forms. Let $T$ be the unique nontrivial $2$-torsion point of $(X, \oplus)$ such that $\ell_O \ell_T$ is nonnegative. Then all the proper faces of $\pos$ are the following:
    \begin{enumerate}
        \item $\face_{A}$ for $A\in X(\RR)$, and  we have $\dim \face_A = 4$;
        \item $\face_{A_1+A_2}$ for $A_1, A_2 \in X(\RR)$, and we have $\dim \face_{A_1+A_2} = 2$.
        \item $\face_{A_1+A_2+B_1}$ with $A_1, A_2 \in X(\RR)$ and $B_1 = O \ominus A_1 \ominus A_2$; in this case $\face_{A_1+A_2+B_1} = \RR_{\ge 0} \cdot \ell^2$ is an extreme ray of the cone of sums of squares and nonnegative quadrics.
        \item $\face_{A_1+A_2+B_2}$ with $A_1, A_2 \in X(\RR)$ and $B_2 = T \ominus A_1 \ominus A_2$; in this case $\face_{A_1+A_2+B_1} = \RR_{\ge 0} \cdot q$ is an extreme ray of the nonnegative quadrics and $q$ is not a sum of squares.
    \end{enumerate}
\end{proposition}

In particular, notice that all the extreme rays of $\pos$ arise from quadrics with the maximal possible number of real zeroes, and all higher dimensional faces have the expected dimension, equal to $6 = \dim \RR[\vb x]_2$ minus two times the number of common zeroes of the quadrics in the face.
Similar elementary considerations can be applied for higher degree nonnegative forms.

Let us also remark that \Cref{eq:divq,eq:divl0l1} can be combined to give explicit certificates of nonnegativity for all the extreme rays of $\pos$. Indeed, if $\RR_{\ge 0} \cdot q$ is an extreme ray of $\pos$, then \Cref{eq:divq,eq:divl0l1} imply that there exists $\alpha \in \RR_{>0}$ such that:
\[
    q = \alpha \, \frac{\ell_{A_1, A_2}^2 \ell_{A_1 \oplus A_2, A_3}^2 \left( x_2^2 \ell_O^2 + (a_2 + a_3 - 2 a_1)\ell_O^2 \ell_{T_1}^2 \right)}{\ell_{\ominus A_1 \ominus A_2, A_1 \oplus A_2}^2 \ell_{T_i}^2 \left( \ell_{T_1}^2+(a_2-a_1)(a_3-a_1)\ell_O^2 \right)} \text{ on } X(\RR)
\]
where $X$ is a totally real plane cubic given by the equation in Weierstrass form \[x_2^{2}x_0-\left(x_1-a_{1}x_0\right)\left(x_1-a_{2}x_0\right)\left(x_1-a_{3}x_0\right)=0\]
The above Artin-type certificates for extreme rays, can be combined in convex combinations to provide explicit certificates of nonnegativity for all nonnegative quadrics $q\in \pos$. We also remark that such certificates do not have the smallest possible degree: indeed, in \cite[Th.~1.1]{blekhermanSharpDegreeBounds2019} is shown the existence of certificates with a denominator of degree two and numerator of degree four. On the other hand, the certificate given in \Cref{eq:divl0l1} is of minimal possible degree.

\section{Proofs of technical lemmas}
\label{app:B}
In this appendix, we prove the technical lemmas in \Cref{sec:waring}.

\begin{proof}[{Proof of \Cref{lem:chiantini}}]
    The statement is equivalent to the fact that the general point of the $d(n+1)$-secant variety of $Z$ is contained in two distinct $d(n+1)$-secant spaces. 
    We can then conclude as in the proof of \cite[Prop.~5.2]{chiantiniConceptKsecantOrder2006}, applied to the elliptic normal curve $Z$.
\end{proof}
\begin{proof}[{Proof of \Cref{lem:chiantini_real}}]
    Let $L = \sum_{i=1}^{d(n+1)} a_i \eval_{A_i} \in \cone_{d(n+1)}(Z) \setminus \partial\cone_{d(n+1)} (Z)$. Then by \Cref{lem:chiantini} $L$ admits another, unique, distinct representation $L = \sum_{i=1}^{d(n+1)} b_i \eval_{B_i}$, with $b_i \in \CC$ and $B_i\in X$. Our goal is to show that $b_i \in \RR_{\ge 0}$ and $B_i \in X(\RR)$.

    In the following, we denote $\sigma(\, \cdot \,)$ the action of complex conjugation. 
    Assume that $\{ \, B_1, \dots , B_{d(n+1)} \, \} \neq \{ \, \sigma({B}_1), \dots , \sigma({B}_{d(n+1)}) \, \}$. Then:
    \[
        L = \sum_{i=1}^{d(n+1)} a_i \, \eval_{A_i} = \sum_{i=1}^{d(n+1)} b_i \, \eval_{B_i} = \sum_{i=1}^{d(n+1)} \sigma({b}_i) \ \eval_{\sigma({B}_i)}
    \]
    are three distinct decompositions, contradicting \Cref{lem:chiantini}.
    
    Therefore we have two cases: either $B_i \in X(\RR)$, or $\sigma({B_i}) = B_j$ for some $j=j(i)$, and in such a case we also have $b_i = \sigma({b_j})$.
    We now show that the second possibility leads to a contradiction. Pick a real linear form $\ell$ vanishing at all the points $\{ \,B_1, \dots , B_{d(n+1)} \, \}$ except $B_i$ and $\sigma({B}_{i})$. This imposes $d(n+1)-2$ conditions in the $d(n+1)$-dimensional space of linear forms on $\nu_{n,d}(X(\RR))$, and thus we can furthermore assume that $\ell(B_i)^2 = -(b_i + \sigma(b_i)) = \sigma (\ell(\sigma (B_i))^2)$. Therefore:
    \begin{align*}
        L(\ell^2) & = \sum_j a_j \ell(A_j)^2 \ge 0 \\
        L(\ell^2) & = b_i \ell(B_i)^2 + \sigma({b}_i) \ell(\sigma({B}_i))^2 = - (b_i + \sigma(b_i))^2 <0
    \end{align*}
    which is a contradiction.

    Therefore $B_i \in X(\RR)$ for all $i$. We now want to show that $b_i \in \RR_{\ge 0}$. For this, pick a real linear form $\ell$ vanishing at all the points $\{ \,B_1, \dots , B_{d(n+1)} \, \}$ except $B_i$. Therefore:
    \[
        0 \le \sum_{j=1}^{d(n+1)} a_j \ell(A_j)^2 = L(\ell^2) = \sum_{j=1}^{d(n+1)} b_j \ell(B_j)^2 = b_i \ell(B_i)^2
    \]
    showing that $b_i \in \RR_{\ge 0}$ and concluding the proof. 
\end{proof}
\begin{proof}[{Proof of \Cref{lem:chiantini_real_special}}]
    Notice that if $L \in \partial \cone_{d(n+1)}(Z)$ then $L$ is a critical value of the map $\psi$ (extended to the full real locus $X(\RR)$) in the proof of \Cref{prop:car_connected}. More precisely, since $L$ is in the boundary, for every $\vb A \in \widehat X(\RR)^{d(n+1)}$ in the preimage of $L$, the differential $\dd \psi_{\vb A}$ is not surjective. Furthermore, since $L \notin \cone_{d(n+1)-1}(Z)$, then every representation of $L$ uses $d(n+1)$ atoms which add to a $2$-torsion point (see again the proof of \Cref{prop:car_connected}).
    
    Consider then, for every $\alpha \in \jac(\RR)_2$, the set of $d(n+1)$-uples which add to $\alpha$:
    \[
        P_\alpha \coloneqq \{ \, \vb A = (A_1, \dots , A_{d(n+1)}) \in \widehat X(\RR)^{d(n+1)} \mid A_1 \oplus \dots \oplus A_{d(n+1)} = \alpha \, \}
    \]
    and the restriction $\psi_\alpha$ of $\psi$ to this subset:
    \begin{align*}
        \psi_\alpha \colon P_\alpha & \longrightarrow \cI(X)_{2d}^\perp \cong R_{2d}^* \\
        \vb A = (A_1, \dots ,  A_{d(n+1)}) & \longmapsto
            \sum_{i=1}^{d(n+1)} \ell_{A_i}^2
    \end{align*}
    As in the proof of \Cref{prop:car_connected}, analyzing the differential of $\psi_\alpha$ we can show that, if the $A_i$'s are distinct, then the codimension of the image of $\dd \psi_{\vb A}$ is one. A dimension count shows then that, if $L \in \partial \cone_{d(n+1)}(Z) \setminus \cone_{d(n+1)-1}(Z)$, the fiber $\psi_\alpha^{-1}(L)$ is finite (or empty). Since every representation of $L \in \partial \cone_{d(n+1)}(Z) \setminus \cone_{d(n+1)-1}(Z)$ uses distinct atoms which add to a $2$-torsion point, we have shown that $L$ admits finitely many representations.

    Finally, by semicontinuity (see e.g. \cite[Theorem~III.12.8]{hartshorneAlgebraicGeometry1977}) we can conclude that $L$ admits at most two decompositions.
\end{proof}

\begin{proof}[{Proof of \Cref{lem:fiber}}]
    Considering the Veronese reembedding $\nu_{n,d}(X)$, we can restict to the case of quadrics, i.e. $d=1$. As is in the proof of \Cref{prop:car_connected}, consider the map 
        \begin{align*}
            \psi \colon \widehat X(\RR)^{n+2} & \longrightarrow \cI(X)_{2}^\perp \cong R_{2}^* \\
            \vb A = (A_1, \dots ,  A_{n+2}) & \longmapsto \sum_{i=1}^{n+2} \ell_{A_i}^2
        \end{align*}  
    and its differential
    \begin{align*}
        \dd \psi_{\vb A} \colon \mathrm{T}_{\vb A} (\widehat X(\RR)^{n+2}) \cong \mathrm{T}_{A_1} \widehat X(\RR) \times \dots \times \mathrm{T}_{A_{n+2}} \widehat X(\RR)  & \longrightarrow \mathrm{T}_{\psi(\vb A)}(\cI(X)_2^\perp) \cong \cI(X)_2^\perp\\
        \vb v = (v_1, \dots v_{n+2}) & \longmapsto 2\sum_{i=1}^{n+2} \ell_{v_i} \ell_{A_i}
     \end{align*}
    From \Cref{cor:cara}, we deduce that $\Im \psi = \pos_{X, 2}^\vee$, which is a $2(n+1)$-dimensional convex cone. We want to show that, for $L \in \interior \pos_{X, 2}^\vee$, $\dim \psi^{-1}(L)$ has dimension two.
    
    As in the proof of \Cref{prop:car_connected}, we notice that $\dd \psi_{\vb A}$ is not surjective if and only if there exists a double vanishing quadratic form $q \in R_{2}$ at $\vb A = (A_1, \dots ,A_{n+2})$. 
    As $\deg \div q = \deg q \cdot \deg X = 2 (n+1)$, this can happen if and only if the $A_i$'s are not distinct and they admit a double vanishing quadraitc form through them. Therefore, if the $A_i$'s are distinct then $\vb A$ is a regular point of $\psi$, and \[\dim \psi^{-1}(\psi(\vb A)) = \dim \widehat X(\RR)^{n+2} - \dim R_2^* = 2\]

    It is then sufficient to show that every linear functional in the relative interior of $\pos_{X, 2}^\vee$ is the image of a regular point of $\psi$. We proceed as in the proof of \Cref{lem:car_ineq}: let $L \in \interior \pos_{X, 2}^\vee$. For all $A\in \widehat{X}(\RR) \cap \SS^{n}$, there exists $\lambda = \lambda(A)\in \RR_{>0}$ such that 
    \[
        L - \lambda^2 \eval_{A} = L - \lambda^2 \ell_A^2 = q- \ell_{\lambda A}^2 \in \partial \pos_{X, 2}^\vee
    \]
    and $L- \ell_{\lambda A}^2$ can be represented in a unique way as a sum of $n+1$ point evaluations (see the proof of \Cref{lem:car_ineq}). Therefore, if the $n+1$ points are distinct and different from $A$, then $L$ is the image of a regular point of $\psi$, as desired. If the $n+1$ points and $A$ are not distinct for all $A \in X(\RR)$, then $L$ would admit a representation as a sum of $n+1$ point evaluations using every $A \in X(\RR)$, contradicting \Cref{lem:chiantini_real,lem:chiantini_real_special}. Thus $L$ has a representation using $n+2$ distinct point evaluations, and it is the image of a regular point of $\psi$, as desired.    
\end{proof}
We note that the proof of \Cref{lem:fiber} can be simplified if $X(\RR)$ is connected. Indeed, in this case all the linear functionals in the relative interior of $\pos_{X,2}^\vee$ are regular values of $\psi$, and we can immediately conclude that all the fibers have dimension two. On the other hand, if $X(\RR)$ is disconnected then $\psi$ has critical values in the interior of $\pos_{X,2}^\vee$, and the proof cannot be easily simplified.
\begin{proof}[{Proof of \Cref{lem:excluding_points}}]
As in the proof of \Cref{lem:fiber}, we can restrict to the case $d=1$.
    Let $\psi$ be the map in the proof of \Cref{lem:fiber}.
    Assume by contradiction that every representation of $L \in \interior \pos_{X, 2}^\vee \subset R_{2}^*$ uses $\eval_{B_j}$ for some $B_j \in \mathcal B$. Then, if we let $\mathcal A_{B_j}(L) \subset (\widehat X(\RR) \cap \SS^n)^{n+1}$ be the $n+1$-tuples of representing atoms of $L$ which include $B_j$, we have:
    \[
        \psi^{-1}(L) =\mathcal A_{B_1}(L) \cup \dots \cup \mathcal A_{B_m}(L) 
    \] 
    We deduce from \Cref{lem:fiber} that there exists $j$ such that $\dim A_{B_j}(L) = 2$.
    In particular, we have a two-dimensional family of representations of the form:
    \[
        L = \lambda \eval_{B_j} + \sum_{i=1}^{n+1} \gamma_{i}(\lambda) \, \eval_{A_i(\lambda)}
    \]
    where $\gamma_i(\lambda) \in \RR_{\ge 0}$, $A_i(\lambda) \in \widehat X(\RR) \cap \SS^n$ and $\lambda\in \RR_{>0}$. This implies that for some $\lambda\in \RR$ there exists a one-dimensional family of representations for $L - \lambda \eval_{B_j}$ using ${n+1}$ point evaluations, in contradiction to \Cref{lem:chiantini_real}. This concludes the proof.
\end{proof}

\hypersetup{
     citecolor=black
}
{
\printbibliography

@article{angeliniRealIdentifiabilityVs2018,
  title = {Real Identifiability vs. Complex Identifiability},
  author = {Angelini, Elena and Bocci, Cristiano and Chiantini, Luca},
  year = {2018},
  journal = {Linear Multilinear Algebra},
  volume = {66},
  number = {6},
  pages = {1257--1267},
  publisher = {Taylor \& Francis},
  issn = {0308-1087},
  doi = {10.1080/03081087.2017.1347137},
  urldate = {2024-01-10}
}

@incollection {angeliniWaringDecompositionsSpecial2023,
    AUTHOR = {Angelini, Elena and Chiantini, Luca and Oneto, Alessandro},
     TITLE = {Waring decompositions of special ternary forms with different
              {H}ilbert functions},
 BOOKTITLE = {Deformation of {A}rtinian algebras and {J}ordan type},
    SERIES = {Contemp. Math.},
    VOLUME = {805},
     PAGES = {77--93},
 PUBLISHER = {Amer. Math. Soc.},
      YEAR = {2024},
      ISBN = {978-1-4704-7356-3; [9781470476656]},
   MRCLASS = {14N07 (13A02 13C40 14C20 14N05 15A69)},
  MRNUMBER = {4802590},
MRREVIEWER = {Zach\ Teitler},
       DOI = {10.1090/conm/805/16127},
       URL = {https://doi.org/10.1090/conm/805/16127},
}

@article{artinUberZerlegungDefiniter1927,
  title = {Uber Die {{Zerlegung}} Definiter {{Funktionen}} in {{Quadrate}}},
  author = {Artin, Emil},
  year = {1927},
  journal = {Abh. Math. Semin. Univ. Hambg.},
  volume = {5},
  number = {1},
  pages = {100--115},
  issn = {1865-8784},
  doi = {10.1007/BF02952513},
  urldate = {2021-11-08}
}

@book{barvinokCourseConvexity2002,
  title = {A Course in Convexity},
  author = {Barvinok, Alexander},
  year = {2002},
  series = {Graduate {{Studies}} in {{Mathematics}}},
  volume = {54},
  publisher = {American Mathematical Society, Providence, RI},
  doi = {10.1090/gsm/054},
  isbn = {978-0-8218-2968-4}
}

@phdthesis{bhardwajNonnegativePolynomialsSums2020,
  title = {Non-Negative {{Polynomials}}, {{Sums}} of {{Squares}} \& {{The Moment Problem}}},
  author = {Bhardwaj, Abhishek},
  year = {2020},
  urldate = {2024-01-10},
  copyright = {http://legaloffice.weblogs.anu.edu.au/content/copyright/},
  langid = {australian},
  school = {The Australian National University}
}

@article{blekhermanLowRankSumofSquaresRepresentations2019,
  title = {Low-{{Rank Sum-of-Squares Representations}} on {{Varieties}} of {{Minimal Degree}}},
  author = {Blekherman, Grigoriy and Plaumann, Daniel and Sinn, Rainer and Vinzant, Cynthia},
  year = {2019},
  journal = {Int. Math. Res. Not.},
  number = {1},
  pages = {33--54},
  issn = {1073-7928},
  doi = {10.1093/imrn/rnx113},
  urldate = {2023-01-23}
}

@article{blekhermanMaximumTypicalGeneric2015,
  title = {On Maximum, Typical and Generic Ranks},
  author = {Blekherman, Grigoriy and Teitler, Zach},
  year = {2015},
  journal = {Math. Ann.},
  volume = {362},
  number = {3},
  pages = {1021--1031},
  issn = {1432-1807},
  doi = {10.1007/s00208-014-1150-3},
  urldate = {2023-04-05},
  langid = {english}
}

@article{blekhermanRealRankRespect2016,
  title = {Real Rank with Respect to Varieties},
  author = {Blekherman, Grigoriy and Sinn, Rainer},
  year = {2016},
  journal = {Linear Algebra Appl.},
  volume = {505},
  pages = {344--360},
  issn = {0024-3795},
  doi = {10.1016/j.laa.2016.04.035},
  urldate = {2023-04-05},
  langid = {english}
}

@book{blekhermanSemidefiniteOptimizationConvex2012,
  title = {Semidefinite {{Optimization}} and {{Convex Algebraic Geometry}}},
  author = {Blekherman, Grigoriy and Parrilo, Pablo A. and Thomas, Rekha R.},
  year = {2012},
  publisher = {{Society for Industrial and Applied Mathematics}},
  doi = {10.1137/1.9781611972290},
  urldate = {2021-11-18},
  isbn = {978-1-61197-228-3}
}

@article{blekhermanSharpDegreeBounds2019,
  title = {Sharp Degree Bounds for Sum-of-Squares Certificates on Projective Curves},
  author = {Blekherman, Grigoriy and Smith, Gregory G. and Velasco, Mauricio},
  year = {2019},
  journal = {J. Math. Pures Appl.},
  volume = {129},
  pages = {61--86},
  issn = {0021-7824},
  doi = {10.1016/j.matpur.2018.12.010},
  urldate = {2023-01-23}
}

@article{blekhermanSumsSquaresVarieties2016a,
  title = {Sums of Squares and Varieties of Minimal Degree},
  author = {Blekherman, Grigoriy and Smith, Gregory G. and Velasco, Mauricio},
  year = {2016},
  journal = {J. Am. Math. Soc.},
  volume = {29},
  number = {3},
  eprinttype = {jstor},
  pages = {893--913},
  publisher = {American Mathematical Society},
  issn = {0894-0347},
  urldate = {2024-01-16}
}

@article{chiantiniConceptKsecantOrder2006,
  title = {On the Concept of K-Secant Order of a Variety},
  author = {Chiantini, Luca and Ciliberto, Ciro},
  year = {2006},
  journal = {J. London Math. Soc.},
  volume = {73},
  number = {2},
  pages = {436--454},
  publisher = {Cambridge University Press},
  issn = {1469-7750, 0024-6107},
  doi = {10.1112/S0024610706022630},
  urldate = {2024-01-10},
  langid = {english}
}

@article{choiExtremalPositiveSemidefinite1977,
  title = {Extremal Positive Semidefinite Forms},
  author = {Choi, Man-Duen and Lam, Tsit-Yuen},
  year = {1977},
  journal = {Math. Ann.},
  volume = {231},
  number = {1},
  pages = {1--18},
  issn = {1432-1807},
  doi = {10.1007/BF01360024},
  urldate = {2024-04-25},
  langid = {english}
}

@book{curtoFlatExtensionsPositive1998,
  title = {Flat {{Extensions}} of {{Positive Moment Matrices}}: {{Recursively Generated Relations}}: {{Recursively Generated Relations}}},
  shorttitle = {Flat {{Extensions}} of {{Positive Moment Matrices}}},
  author = {Curto, Ra{\'u}l E. and Fialkow, Lawrence A.},
  year = {1998},
  publisher = {American Mathematical Soc.},
  googlebooks = {gpfTCQAAQBAJ},
  isbn = {978-0-8218-0869-6},
  langid = {english}
}

@article{didioMultidimensionalTruncatedMoment2018,
  title = {The Multidimensional Truncated Moment Problem: {{Carath{\'e}odory}} Numbers},
  shorttitle = {The Multidimensional Truncated Moment Problem},
  author = {{di Dio}, Philipp J. and Schm{\"u}dgen, Konrad},
  year = {2018},
  journal = {J. Math. Anal. Appl.},
  volume = {461},
  number = {2},
  pages = {1606--1638},
  issn = {0022-247X},
  doi = {10.1016/j.jmaa.2017.12.021},
  urldate = {2024-01-10}
}

@article{didioMultidimensionalTruncatedMoment2021a,
  title = {The Multidimensional Truncated Moment Problem: {{Carath{\'e}odory}} Numbers from {{Hilbert}} Functions},
  shorttitle = {The Multidimensional Truncated Moment Problem},
  author = {{di Dio}, Philipp J. and Kummer, Mario},
  year = {2021},
  journal = {Math. Ann.},
  volume = {380},
  number = {1},
  pages = {267--291},
  issn = {1432-1807},
  doi = {10.1007/s00208-021-02166-x},
  urldate = {2023-04-05},
  langid = {english}
}

@article{eisenbudCayleyBacharachTheoremsConjectures1996,
  title = {Cayley-{{Bacharach}} Theorems and Conjectures},
  author = {Eisenbud, David and Green, Mark and Harris, Joe},
  year = {1996},
  journal = {Bull. Amer. Math. Soc.},
  volume = {33},
  number = {3},
  pages = {295--324},
  issn = {0273-0979, 1088-9485},
  doi = {10.1090/S0273-0979-96-00666-0},
  urldate = {2024-05-22},
  langid = {english}
}

@article{eisenbudProjectiveGeometryGale2000,
  title = {The {{Projective Geometry}} of the {{Gale Transform}}},
  author = {Eisenbud, David and Popescu, Sorin},
  year = {2000},
  journal = {J. Algebra},
  volume = {230},
  number = {1},
  pages = {127--173},
  issn = {0021-8693},
  doi = {10.1006/jabr.1999.7940},
  urldate = {2024-03-08}
}

@article{fialkowSolutionTruncatedMoment2011,
  title = {Solution of the Truncated Moment Problem with Variety y = X{$^{3}$}},
  author = {Fialkow, Lawrence A.},
  year = {2011},
  journal = {Trans. Amer. Math. Soc.},
  volume = {363},
  number = {6},
  eprinttype = {jstor},
  pages = {3133--3165},
  publisher = {American Mathematical Society},
  issn = {0002-9947},
  urldate = {2024-01-10}
}

@book{fultonAlgebraicCurves1989,
  title = {Algebraic Curves},
  author = {Fulton, William},
  year = {1989},
  series = {Advanced {{Book Classics}}},
  publisher = {Addison-Wesley Publishing Company, Advanced Book Program, Redwood City, CA},
  isbn = {978-0-201-51010-2}
}

@article{geyerUeberlagerungenBerandeterKleinscher1977,
  title = {{{\"U}berlagerungen berandeter Kleinscher Fl{\"a}chen}},
  author = {Geyer, W. D. and Martens, G.},
  year = {1977},
  journal = {Math. Ann.},
  volume = {228},
  number = {2},
  pages = {101--111},
  issn = {1432-1807},
  doi = {10.1007/BF01351166},
  urldate = {2024-01-10},
  langid = {ngerman}
}

@article{grossRealAlgebraicCurves1981a,
  title = {Real Algebraic Curves},
  author = {Gross, Benedict H. and Harris, Joe},
  year = {1981},
  journal = {Ann. Sci. Ec. Norm. Super.},
  volume = {14},
  number = {2},
  pages = {157--182},
  issn = {1873-2151},
  doi = {10.24033/asens.1401},
  urldate = {2023-12-18},
  langid = {english}
}

@book{hartshorneAlgebraicGeometry1977,
  title = {Algebraic {{Geometry}}},
  author = {Hartshorne, Robin},
  year = {1977},
  series = {Graduate {{Texts}} in {{Mathematics}}},
  volume = {52},
  publisher = {Springer},
  address = {New York, NY},
  doi = {10.1007/978-1-4757-3849-0},
  urldate = {2024-01-30},
  isbn = {978-1-4419-2807-8 978-1-4757-3849-0}
}

@article{hilbertUeberDarstellungDefiniter1888,
  title = {Ueber Die {{Darstellung}} Definiter {{Formen}} Als {{Summe}} von {{Formenquadraten}}},
  author = {Hilbert, D.},
  year = {1888},
  journal = {Math. Ann.},
  volume = {32},
  pages = {342--350},
  doi = {10.1007/BF01443605}
}

@article{huismanGeometryAlgebraicCurves2001,
  title = {On the Geometry of Algebraic Curves Having Many Real Components.},
  author = {Huisman, J.},
  year = {2001},
  journal = {Rev. Mat. Complut.},
  volume = {14},
  number = {1},
  pages = {83--92},
  issn = {1139-1138},
  urldate = {2023-12-18},
  langid = {english}
}

@book{iarrobinoPowerSumsGorenstein1999,
  title = {Power {{Sums}}, {{Gorenstein Algebras}}, and {{Determinantal Loci}}},
  author = {Iarrobino, Anthony and Kanev, Vassil},
  year = {1999},
  publisher = {Springer-Verlag},
  address = {Berlin Heidelberg},
  doi = {10.1007/BFb0093426},
  urldate = {2021-10-07},
  isbn = {978-3-540-66766-7},
  langid = {english}
}

@article{kunertExtremePositiveTernary2018,
  title = {Extreme Positive Ternary Sextics},
  author = {Kunert, Aaron and Scheiderer, Claus},
  year = {2018},
  journal = {Trans. Amer. Math. Soc.},
  volume = {370},
  number = {6},
  pages = {3997--4013},
  issn = {0002-9947, 1088-6850},
  doi = {10.1090/tran/7076},
  urldate = {2023-05-15},
  langid = {english}
}

@phdthesis{kunertFacialStructureCones2014,
  title = {Facial {{Structure}} of {{Cones}} of Non-Negative {{Forms}}},
  author = {Kunert, Aaron},
  year = {2014},
  langid = {australian}
}

@article{langeHigherSecantVarieties1984,
  title = {Higher Secant Varieties of Curves and the Theorem of {{Nagata}} on Ruled Surfaces},
  author = {Lange, Herbert},
  year = {1984},
  journal = {Manuscripta Math.},
  volume = {47},
  number = {1},
  pages = {263--269},
  issn = {1432-1785},
  doi = {10.1007/BF01174597},
  urldate = {2024-02-01},
  langid = {english}
}

@book{mangolteRealAlgebraicVarieties2020a,
  title = {Real Algebraic Varieties},
  author = {Mangolte, Fr{\'e}d{\'e}ric},
  year = {2020},
  series = {Springer {{Monographs}} in {{Mathematics}}},
  publisher = {Springer, Cham},
  doi = {10.1007/978-3-030-43104-4},
  isbn = {978-3-030-43104-4 978-3-030-43103-7}
}

@book{marshallPositivePolynomialsSums2008a,
  title = {Positive Polynomials and Sums of Squares},
  author = {Marshall, Murray},
  year = {2008},
  series = {Mathematical {{Surveys}} and {{Monographs}}},
  volume = {146},
  publisher = {American Mathematical Society, Providence, RI},
  doi = {10.1090/surv/146},
  isbn = {978-0-8218-4402-1}
}

@article{monnierDivisorsRealCurves2003,
  title = {Divisors on Real Curves},
  author = {Monnier, Jean-Philippe},
  year = {2003},
  journal = {Adv. Geom.},
  volume = {3},
  number = {3},
  pages = {339--360},
  publisher = {De Gruyter},
  issn = {1615-7168},
  doi = {10.1515/advg.2003.019},
  urldate = {2023-12-18},
  chapter = {Advances in Geometry},
  copyright = {De Gruyter expressly reserves the right to use all content for commercial text and data mining within the meaning of Section 44b of the German Copyright Act.},
  langid = {english}
}

@article{mourrainPolynomialExponentialDecomposition2018a,
  title = {Polynomial--{{Exponential Decomposition From Moments}}},
  author = {Mourrain, Bernard},
  year = {2018},
  journal = {Found. Comput. Math.},
  volume = {18},
  number = {6},
  pages = {1435--1492},
  issn = {1615-3383},
  doi = {10.1007/s10208-017-9372-x},
  urldate = {2024-04-24},
  langid = {english}
}

@article{qiSemialgebraicGeometryNonnegative2016,
  title = {Semialgebraic {{Geometry}} of {{Nonnegative Tensor Rank}}},
  author = {Qi, Yang and Comon, Pierre and Lim, Lek-Heng},
  year = {2016},
  journal = {SIAM J. Matrix Anal. Appl.},
  volume = {37},
  number = {4},
  pages = {1556--1580},
  publisher = {{Society for Industrial and Applied Mathematics}},
  issn = {0895-4798},
  doi = {10.1137/16M1063708},
  urldate = {2024-02-20}
}

@incollection{reznickConcreteAspectsHilbert2000,
  title = {Some Concrete Aspects of {{Hilbert}}'s 17th {{Problem}}},
  booktitle = {Real Algebraic Geometry and Ordered Structures ({{Baton Rouge}}, {{LA}}, 1996)},
  author = {Reznick, Bruce},
  year = {2000},
  series = {Contemp. {{Math}}.},
  volume = {253},
  pages = {251--272},
  publisher = {American Mathematical Society},
  doi = {10.1090/conm/253/03936}
}

@article{rienerOptimizationApproachesQuadrature2018a,
  title = {Optimization Approaches to Quadrature: {{New}} Characterizations of {{Gaussian}} Quadrature on the Line and Quadrature with Few Nodes on Plane Algebraic Curves, on the Plane and in Higher Dimensions},
  shorttitle = {Optimization Approaches to Quadrature},
  author = {Riener, Cordian and Schweighofer, Markus},
  year = {2018},
  journal = {J. Complex.},
  volume = {45},
  pages = {22--54},
  issn = {0885-064X},
  doi = {10.1016/j.jco.2017.10.002},
  urldate = {2024-02-01}
}

@book{rockafellarConvexAnalysis1970a,
  title = {Convex Analysis},
  author = {Rockafellar, R. Tyrrell},
  year = {1970},
  series = {Princeton {{Mathematical Series}}},
  volume = {No. 28},
  publisher = {Princeton University Press, Princeton, NJ}
}

@article{scheidererSumsSquaresRegular2000,
  title = {Sums of Squares of Regular Functions on Real Algebraic Varieties},
  author = {Scheiderer, Claus},
  year = {2000},
  journal = {Trans. Amer. Math. Soc.},
  volume = {352},
  number = {3},
  pages = {1039--1069},
  issn = {0002-9947, 1088-6850},
  doi = {10.1090/S0002-9947-99-02522-2},
  urldate = {2020-12-04},
  langid = {english}
}

@book{schmudgenMomentProblem2017a,
  title = {The {{Moment Problem}}},
  author = {Schm{\"u}dgen, Konrad},
  year = {2017},
  series = {Graduate {{Texts}} in {{Mathematics}}},
  volume = {277},
  publisher = {Springer International Publishing},
  address = {Cham},
  doi = {10.1007/978-3-319-64546-9},
  urldate = {2024-01-16},
  isbn = {978-3-319-64545-2 978-3-319-64546-9}
}

@phdthesis{schulzeConesLocallyNonNegative2021,
  title = {Cones of {{Locally Non-Negative Polynomials}}},
  author = {Schulze, Richard Christoph},
  year = {2021},
  langid = {australian}
}

@book{shafarevichBasicAlgebraicGeometry2013,
  title = {Basic {{Algebraic Geometry}} 1: {{Varieties}} in {{Projective Space}}},
  shorttitle = {Basic {{Algebraic Geometry}} 1},
  author = {Shafarevich, Igor R.},
  year = {2013},
  edition = {3},
  publisher = {Springer-Verlag},
  address = {Berlin Heidelberg},
  doi = {10.1007/978-3-642-37956-7},
  urldate = {2019-12-10},
  isbn = {978-3-642-37955-0},
  langid = {english}
}

@article{vinnikovSelfadjointDeterminantalRepresentations1993,
  title = {Self-Adjoint Determinantal Representations of Real Plane Curves},
  author = {Vinnikov, Victor},
  year = {1993},
  journal = {Math. Ann.},
  volume = {296},
  number = {1},
  pages = {453--479},
  issn = {1432-1807},
  doi = {10.1007/BF01445115},
  urldate = {2023-04-06},
  langid = {english}
}

@article{zalarTruncatedHamburgerMoment2021,
  title = {The {{Truncated Hamburger Moment Problems}} with {{Gaps}} in the {{Index Set}}},
  author = {Zalar, Alja{\v z}},
  year = {2021},
  journal = {Integr. Equ. Oper. Theory},
  volume = {93},
  number = {3},
  pages = {22},
  issn = {1420-8989},
  doi = {10.1007/s00020-021-02628-6},
  urldate = {2024-01-10},
  langid = {english}
}

@article{zalarTruncatedMomentProblem2022,
  title = {The Truncated Moment Problem on the Union of Parallel Lines},
  author = {Zalar, Alja{\v z}},
  year = {2022},
  journal = {Linear Algebra Appl.},
  volume = {649},
  pages = {186--239},
  issn = {0024-3795},
  doi = {10.1016/j.laa.2022.05.008},
  urldate = {2024-01-10}
}

@article{zalarTruncatedMomentProblem2023,
  title = {The Truncated Moment Problem on Curves y = q(x) and Yx{$\ell$} = 1},
  author = {Zalar, A.},
  year = {2023},
  journal = {Linear Multilinear Algebra},
  pages = {1--45},
  publisher = {Taylor \& Francis},
  issn = {0308-1087},
  doi = {10.1080/03081087.2023.2212316},
  urldate = {2024-01-10}
}

@misc{baldi2025totallyrealdivisorscurves,
      title={Totally real divisors on curves}, 
      author={Lorenzo Baldi and Mario Kummer and Daniel Plaumann},
      year={2025},
      eprint={2509.07544},
      archivePrefix={arXiv},
      primaryClass={math.AG},
      url={https://arxiv.org/abs/2509.07544}, 
}

@misc{baldi2026stubbornpolynomials,
      title={Stubborn Polynomials}, 
      author={Lorenzo Baldi and Grigoriy Blekherman and Khazhgali Kozhasov and Daniel Plaumann and Bruce Reznick and Rainer Sinn},
      year={2026},
      eprint={2602.01191},
      archivePrefix={arXiv},
      primaryClass={math.AG},
      url={https://arxiv.org/abs/2602.01191}, 
}

@misc{kummer2026nonnegativepolynomialsgeneralizedellipticv1,
      title={Positive polynomials and the truncated moment problem on plane cubics}, 
      author={Mario Kummer and Aljaž Zalar},
      year={2025},
      eprint={2508.13850v1},
      archivePrefix={arXiv},
      primaryClass={math.FA},
      url={https://arxiv.org/abs/2508.13850v1}, 
}

@misc{kummer2026nonnegativepolynomialsgeneralizedellipticv3,
      title={Non-negative polynomials on generalized elliptic curves}, 
      author={Mario Kummer and Aljaž Zalar},
      year={2026},
      eprint={2508.13850v3},
      archivePrefix={arXiv},
      primaryClass={math.FA},
      url={https://arxiv.org/abs/2508.13850v3}, 
}
}
\end{document}